%% file: ACA_2105_v25.tex
\documentclass[graybox]{svmult}

\include{macrosB}
\include{newpack3}

\usepackage{mathptmx}       
\usepackage{helvet}         
\usepackage{courier}        
\usepackage{type1cm}        
\usepackage{makeidx}         
\usepackage{graphicx}        
\usepackage{multicol}        
\usepackage[bottom]{footmisc}

\makeindex             

\def\calA{{\cal A}}
\def\calD{{\cal D}}

\def\calU{{\cal U}}
\def\pol#1{\langle #1 \rangle}
\def\ser#1{\langle\!\langle #1 \rangle\!\rangle}
\def\AX{A \langle X \rangle}
\def\QX{\Q\langle X \rangle}
\def\QY{\Q\langle Y \rangle}

\def\scal#1#2{\langle #1\bv#2 \rangle}
\def\bv{\mid}
\def\bbv{\bv\!\bv}
\def\abs#1{\bv\!#1\!\bv}
\def\LAX{\Lie_{A} \langle X \rangle}
\def\LQX{\Lie_{\Q} \langle X \rangle}

\def\1{{1_{\Omega}}}

\newcounter{per1}
\begingroup
\def\2#1{\ifnum#1<10 0\fi\the#1}
\xdef\isodayandtime{
{\2\day-\2\month-\the\year\space\2{\count0}:%
\2{\count2}}}
\endgroup

\begin{document}

\title*{Mathematical renormalization in quantum electrodynamics via noncommutative generating series}
\author{G. H. E. Duchamp - Hoang Ngoc Minh - Q. H. Ngo - K. Penson - P. Simonnet}
\institute{G. H. E. Duchamp \at LIPN - UMR 7030, CNRS, 93430 Villetaneuse, France, \email{gheduchamp@gmail.com}
\and Hoang Ngoc Minh \at Universit\'e Lille II,  59024 Lille, France, \email{hoang@univ.lille2.fr}
\and Q. H. Ngo \at Universit\'e Paris XIII, 93430 Villetaneuse, France, \email{quochoan\_ngo@yahoo.com.vn}
\and K. Penson \at Universit\'e Paris VI - LPTMC, 75252 Paris Cedex 05, France, \email{penson@lptmc.jussieu.fr}
\and P. Simonnet \at Universit\'e de Corse, 20250 Corte, France, \email{simonnet@univ-corse.fr}}
%
%
\maketitle


\abstract{ In order to push the study of solutions of nonlinear differential equations involved in quantum electrodynamics\footnote{
The present work is part of a series of papers devoted to the study of the renormalization
of divergent polyzetas (at positive and at negatice indices) via the factorization
of the non commutative generating series of polylogarithms and of harmonic sums
and via the effective construction of pairs of bases in duality in $\varphi$-deformed shuffle algebras.
It is a sequel of \cite{BDHHT} and its content was presented in several seminars and meetings,
including the 66th and 74th S\'eminaire Lotharingien de Combinatoire.}, we focus here on combinatorial aspects of their  renormalization at $\{0,1,+\infty\}$.}
\section{Introduction}
During the last century, the functional expansions were common in physics as well as in engineering
and have been developped, by Tomonaga, Schwinger and Feynman  \cite{dyson},
to represent the dynamical systems in quantum electrodynanics.
The main difficulty of this approach is the divergence of these expansions
at the singularity $0$ or at $+\infty$ (see \cite{BBM}) and leads to the problems of {\it regularization} and {\it renormalization}
which can be solved by combinatorial technics~: Feynman diagrams \cite{FH} and their siblings \cite{interface,lemurakami},
noncommutative formal power series \cite{fliess1}, trees \cite{CK}.

Recently, in the same vein, and based, on the one hand, on the shuffle and quasi-shuffle bialgebras \cite{BDHHT},
the combinatorics of noncommutative formal power series
was intensively amplified for the asymptotic analysis of dynamical systems with three regular singularities
in\footnote{Any differential equation with singularities in $\{a,b,c\}$ can be changed into a differential equation with singularities
in $\{0,1,+\infty\}$ via an homographic transformation.} $\{0,1,+\infty\}$ ;
and, on the other hand with the monodromy and the Galois differential group of the Knizhnik-Zamolodchikov equation $KZ_3$ \cite{orlando,acta} {\it i.e.},
the following noncommutative evolution equation\footnote{$x_0:=\phantom{-}{t_{1,2}}/{2\mathrm{i}\pi}$ and $x_1:=-{t_{2,3}}/{2\mathrm{i}\pi}$ are noncommuative variables
and $t_{1,2},t_{2,3}$ belong to $\calT_3=\{t_{1,2},t_{1,3},t_{2,3}\}$ satisfying the infinitesimal $3$-braid relations, {\it i.e.} $[t_{1,3},t_{1,2}+t_{2,3}]=[t_{2,3},t_{1,2}+t_{1,3}]=0$.}
\begin{eqnarray*}
\frac{dG(z)}{dz}=\biggl(\Frac{x_0}{z}+\Frac{x_1}{1-z}\biggr)G(z),
\end{eqnarray*}
the monoidal factorization facilitates mainly the renormalization and the computation of the associators\footnote{
They were introduced in QFT by Drinfel'd and it plays an important role for the still open problem
of the effective determination of the polynomial invariants of knots and links via Kontsevich's integral (see \cite{cartier,lemurakami})
and  $\Phi_{KZ}$, was obtained firstly,  in \cite{lemurakami}, with explicit coefficients which are polyzetas and regularized polyzetas
(see \cite{acta,VJM} for the computation of the other involving {\it only} convergent polyzetas as local coordinates,
and for algorithms regularizing divergent polyzetas).}
via the universal one, {\it i.e.}  $\Phi_{KZ}$ of Drinfel'd \cite{acta}.

In fact, these associators are noncommutative formal power series on two variables and regularize
the Chen generating series of the differential forms admitting singularities at $0$ or at $1$ along the integration paths
on the universal covering of $\C$ without points $0$ and $1$ ({\it i.e.} $\widetilde{\C\setminus\{0,1\}}$).
Their coefficients are, up to a multiple of powers of $2\mbox{i}\pi$, polynomial on polyzetas,
{\it i.e.} the following real numbers\footnote{$s_1+\ldots+s_r$ is the {\it weight} of $\zeta(s_1,\ldots,s_r)$ ,
{\it i.e.} the weight of the composition $(s_1,\ldots,s_r)$.} \cite{lemurakami,hoffman0,zagier}
\begin{eqnarray*}
\zeta(s_1,\ldots,s_r)=\sum_{n_1>\ldots>n_r>0}\frac1{n_1^{s_1}\ldots n_r^{s_r}},&\mbox{for}&
r\ge1,s_1\ge2,s_2,\ldots,s_r\ge1,
\end{eqnarray*}
and these numbers admit a natural structure of algebra over the rational numbers deduced from the combinatorial aspects
of the shuffle and quasi-shuffle Hopf algebras. It is conjectured that this algebra is $\N$-graded\footnote{One of us proposed a proof in \cite{acta,VJM}.}.
More precisely, for $s_1\ge2,s_2,\ldots,s_r\ge1$, the polyzeta $\zeta(s_1,\ldots,s_r)$ can be obtained as the limit of
the polylogarithm $\Li_{s_1,\ldots,s_r}(z)$, for $z\rightarrow 1$, and of
the harmonic sum $\H_{s_1,\ldots,s_r}(N)$, for $N\rightarrow+\infty$~:
\begin{eqnarray*}
\Li_{s_1,\ldots,s_r}(z)=\sum_{n_1>\ldots>n_r>0}\frac{z^{n_1}}{n_1^{s_1}\ldots n_r^{s_r}}
&\mbox{and}&
\H_{s_1,\ldots,s_r}(N)=\sum_{n_1>\ldots>n_r>0}^N\frac1{n_1^{s_1}\ldots n_r^{s_r}}.
\end{eqnarray*}
Then, by a theorem of Abel, one has
\begin{eqnarray*}
\zeta(s_1,\ldots,s_r)=\lim_{z\rightarrow 1}\Li_{s_1,\ldots,s_r}(z)
=\lim_{N\rightarrow +\infty}\H_{s_1,\ldots,s_r}(N).
\end{eqnarray*}

Since the algebras of polylogarithms and of harmonic sums are isomorphic
to the shuffle algebra $(\QX,\shuffle,1_{X^*})$ and quasi-shuffle algebra $(\QY,\stuffle,1_{Y^*})$ respectively
both admitting the Lyndon words $\Lyn X$ over $X=\{x_0,x_1\}$ and $\Lyn Y$ over $Y=\{y_i\}_{i\ge1}$,
as transcendence bases (recalled in Section \ref{shuffles}) then, by using
\begin{itemize}
\item The (one-to-one) correspondence between the combinatorial compositions, the words\footnote{Here, $X^*$ (resp. $Y^*$)
is the monoid generated by $X$ (resp. $Y$) and its neutral element of  is denoted by $1_{X^*}$ (resp. $1_{Y^*}$).}
in $Y^*$ and the words in $X^*x_1+ 1_{X^*}$, {\it i.e.}\footnote{Here,
$\pi_Y$ is the adjoint of $\pi_X$ for the canonical scalar products where $\pi_X$ is the morphism of AAU
$\ncp{k}{Y}\rightarrow \ncp{k}{X}$ defined by $\pi_X(y_k)=x_0^{k-1}x_1$.}
\begin{eqnarray}\label{correspondence}
(\{1\}^k,s_{k+1},\ldots,s_r)
\leftrightarrow y_1^ky_{s_{k+1}}\ldots y_{s_r}
\mathop{\rightleftharpoons}_{\pi_Y}^{\pi_X}x_1^kx_0^{s_{k+1}-1}x_1\ldots x_0^{s_r-1}x_1.
\end{eqnarray}
\item The ordering $x_1 \succ x_0$ and $y_1 \succ y_2 \succ \ldots$ over $X$ and $Y$ respectively.
\item The transcendence base $\{S_l\}_{l\in\Lyn X}$ (resp. $\{\Sigma_l\}_{l\in\Lyn Y}$) of $(\QX,\shuffle,1_{X^*})$ (resp. $(\QY,\stuffle,1_{Y^*})$)
in duality\footnote{in a more precise way the $S$ and $\Sigma$ are the ``Lyndon part'' of the dual bases of the PBW expansions of the $P$
and the $\Pi$ respectively.} with $\{P_l\}_{l\in\Lyn X}$ (resp. $\{\Pi_l\}_{l\in\Lyn Y}$), a base of the Lie algebra of primitive elements
of the bialgebra\footnote{$\epsilon$ is the ``constant term'' character.} $\calH_{\shuffle}=(\QX,{\tt conc},1_{X^*},\Delta_{\shuffle},\epsilon)$
(resp. $\calH_{\stuffle}=(\QY,{\tt conc},1_{Y^*},\Delta_{\stuffle},\epsilon)$) to factorize
the following noncommutative generating series of polylogarithms, hormanic sums and polyzetas
\begin{eqnarray*}
\L=\prod_{l\in\Lyn X}\exp(\Li_{S_l}P_l)&\mbox{and}&\H=\Prod_{l\in\Lyn Y}\exp(\H_{\Sigma_l}\Pi_l),\\
Z_{\minishuffle}=\prod_{l\in\Lyn X,l\neq x_0,x_1}\exp(\zeta(S_l)P_l)&\mbox{and}&Z_{\ministuffle}=\prod_{l\in\Lyn Y,l\neq y_1}\exp(\zeta(\Sigma_l)\Pi_l),
\end{eqnarray*}
\end{itemize}
we then obtain two formal power series over $Y$, $Z_1$ and $Z_2$, such that
\begin{eqnarray*}
\Lim_{z\rightarrow1}\exp\biggl[y_1\log\frac1{1-z}\biggr]\pi_Y\L(z)=Z_1,&&
\Lim_{N\rightarrow\infty}\exp\biggl[\Sum_{k\ge1}\H_{y_k}(N)\frac{(-y_1)^k}{k}\bigg]\H(N)=Z_2.
\end{eqnarray*}
Moreover, $Z_1,Z_2$ are equal and stand for the noncommutative generating series of
$\{\zeta(w)\}_{w\in Y^*- y_1Y^*}$, or $\{\zeta(w)\}_{w\in x_0X^*x_1}$, as one has $Z_1=Z_2=\pi_YZ_{\minishuffle}$ \cite{cade,acta,VJM}.
This allows, by extracting the coefficients of the noncommutative generating series,
to explicit the counter-terms eliminating the divergence of $\{\Li_w\}_{w\in x_1X^*}$ and of  $\{\H_w\}_{w\in y_1Y^*}$
and this leads naturally to an equation connecting algebraic structures
\begin{eqnarray}\label{pont}
\Prod_{l\in\Lyn Y,l\neq y_1}^{\searrow}\exp(\zeta(\Sigma_l)\Pi_l)
=\exp\biggl[\sum_{k\ge2} - \zeta(k)\Frac{(-y_1)^k}{k}\biggr]
\pi_Y\Prod_{l\in\Lyn X,l\neq x_0,x_1}^{\searrow}\exp(\zeta(S_l)P_l).
\end{eqnarray}
Identity (\ref{pont}) allows to compute the Euler-MacLaurin constants and the Hadamard finite parts
associated to divergent polyzetas $\{\zeta(w)\}_{w\in y_1Y^*}$ and,
by identifying local coordinates, to describe the graded core of $\ker\zeta$ by its {\it algebraic} generators.

In this paper, we will focus on the approach by noncommutative formal power series, adapted from \cite{fliess0,fliess1},
and explain how the results of \cite{cade,acta,VJM}, allow to study the combinatorial aspects of the renormalization
at the singularities in $\{0,1,+\infty\}$ of the solutions of linear differential equations (see Example \ref{Hypergeometric} below)
as well as the solutions of nonlinear differential equations (see Examples \ref{oscillator} and \ref{Duffing} below)
described in Section \ref{polysystem} and involved in quantum electrodynamics.

\begin{example}[Hypergeometric equation]\label{Hypergeometric}
\small{Let $t_0,t_1,t_2$ be parameters and 
\begin{eqnarray*}
z(1-z)\ddot y(z)+[t_2-(t_0+t_1+1)z] \dot y(z)-t_0t_1 y(z)=0.
\end{eqnarray*}
Let $q_1(z)=-y(z)$ and $q_2(z)=(1-z)\dot y(z)$. One has
\begin{eqnarray*}
\begin{pmatrix}\dot q_1\cr\dot q_2\end{pmatrix}
=\left(\frac{M_0}z+\frac{M_1}{1-z}\right)
\begin{pmatrix}q_1\cr q_2\end{pmatrix},
\end{eqnarray*}
where $M_0$ and $M_1$ are the following matrices
\begin{eqnarray*}
M_0 = -\begin{pmatrix}0&0\cr t_0t_1& t_2\end{pmatrix}&\mbox{and}&
M_1 = -\begin{pmatrix}0&1\cr0&t_2-t_0-t_1\end{pmatrix}.
\end{eqnarray*}
Or equivalently,
\begin{eqnarray*}
\dot q(z) = A_0(q)\frac{1}z+A_1(q)\frac{1}{1-z}&\mbox{and}&y(z)=-q_1(z),
\end{eqnarray*}
where $A_0$ and $A_1$ are the following parametrized linear vector fields
\begin{eqnarray*}
A_0 = - (t_0t_1q_1+t_2q_2)\frac{\partial}{\partial q_2}
&\mbox{and}
&A_1 = - q_1\frac{\partial}{\partial q_1}-(t_2-t_0-t_1)q_2\frac{\partial}{\partial q_2}.
\end{eqnarray*}
acting by 
\begin{eqnarray*}
\frac{\partial}{\partial q_1}(q)=\frac{\partial}{\partial q_1}\begin{pmatrix}q_1\cr q_2\end{pmatrix}=\begin{pmatrix} 1\cr 0\end{pmatrix}&\mbox{and}&
\frac{\partial}{\partial q_2}(q)=\frac{\partial}{\partial q_2}\begin{pmatrix}q_1\cr q_2\end{pmatrix}=\begin{pmatrix} 0\cr 1\end{pmatrix}.
\end{eqnarray*}
}
\end{example}

\begin{example}[Harmonic oscillator]\label{oscillator}
\small{Let $k_1,k_2$ be parameters and 
\begin{eqnarray*}
\dot y(z)+k_1y(z)+k_2y^2(z)=u_1(z).
\end{eqnarray*}
which can be represented, with the same formalism as above, by the following state equations
\begin{eqnarray*}
\dot q(z)=A_0(q)+A_1(q)u_1(z)&\mbox{and}&y(z)=q(z),
\end{eqnarray*}
where $A_0$ and $A_1$ are the following vector fields
\begin{eqnarray*}
A_0 = -(k_1q+k_2q^2)\frac{\partial}{\partial q}&\mbox{and}& A_1 = \frac{\partial}{\partial q}.
\end{eqnarray*}
}
\end{example}

\begin{example}[Duffing's equation]\label{Duffing}
\small{Let $a,b,c$ be parameters and 
\begin{eqnarray*}
\ddot y(z)+a\dot y(z)+by(z)+cy^3(z)=u_1(z).
\end{eqnarray*}
which can be represented by the following state equations
\begin{eqnarray*}
\dot q(z)=A_0(q)+A_1(q)u_1(z),&\mbox{and}&y(z)=q_1(z),
\end{eqnarray*}
where $A_0$ and $A_1$ are the following vector fields
\begin{eqnarray*}
A_0 = -(aq_2+b^2q_1+cq_1^3)\frac{\partial}{\partial q_2}+q_2\frac{\partial}{\partial q_1}
&\mbox{and}&
A_1 = \frac{\partial}{\partial q_2}.
\end{eqnarray*}}
\end{example}

\begin{example}[Van der Pol oscillator]\label{vanderpol}
\small{Let $\gamma,g$ be parameters and
\begin{eqnarray*}
\partial^2_zx(z)-\gamma[1+x(z)^2]\partial_zx(z)+x(z)=g\cos(\omega z)
\end{eqnarray*}
which can be tranformed into (with $C$ is some constant of integration)
\begin{eqnarray*}
\partial_zx(z)=\gamma[1+x(z)^2/3]x(z)-\int_{z_0}^z x(s)ds+\frac g{\omega}\sin(\omega z)+C.
\end{eqnarray*}
Setting $y = \int_{z_0}^z x(s)ds$ and $u_1(z)= g\sin(\omega z)/{\omega}+C$, it leads then to
\begin{eqnarray*}
\partial_z^2y(z)=\gamma[\partial_zy(z)+(\partial_zy(z))^3/3]-y(z)+u_1(z)
\end{eqnarray*}
which can be represented by the following state equations (with $n=2$)
\begin{eqnarray*}
\partial_zq(z)=[A_0 u_0(z)+ A_1 u_1(z)](q)&\mbox{and}&y(z)=q_1(z)
\end{eqnarray*}
where $A_0$ and $A_1$ are the following vector fields
\begin{eqnarray*}
A_1=\frac{\partial}{\partial q_2}&\mbox{and}&
A_0=[\gamma(q_2+q_2^3/3)-q_1]\frac{\partial}{\partial q_2}+q_2\frac{\partial}{\partial q_1}.
\end{eqnarray*}}
\end{example}

This approach by noncommutative formal power series is adequate for studying the algebraic combinatorial aspects of the asymptotic analysis 
at the singularities in $\{0,1,+\infty\}$, of the nonlinear dynamical systems described in Section \ref{polysystem} because
\begin{itemize}
\item The polylogarithms form a basis of an infinite dimensional universal Picard-Ves\-siot extension by means of these differential equations \cite{orlando,DDMS}
and their algebra, isomorphic to the shuffle algebra, admits $\{\Li_l\}_{l\in\Lyn X}$ as a transcendence basis,
\item The harmonic sums generate the coefficients of the ordinary Taylor expansions of their solutions (when these expansions exist) \cite{cade}
and their algebra is isomorphic to the quasi-shuffle algebra admitting $\{\H_l\}_{l\in\Lyn Y}$ as a transcendence basis,
\item The polyzetas do appear as the fondamental arithmetical constants involved in the computations of the monodromies \cite{FPSAC98,SLC43},
the Kummer type functional equations \cite{FPSAC99,SLC43}, the asymptotic expansions of solutions \cite{cade,acta}
and their algebra is freely generated by the polyzetas encoded by irreducible Lyndon words \cite{acta}.
\end{itemize}

Hence, a lot of algorithms can be deduced from these facts and more general studies will be proceeded in \cite{BDHHT,DDMS}.
The organisation of this paper is the following
\begin{itemize}
\item In Section \ref{Background}, we will give algebraic and analytic foundations, {\it i.e.} the combinatorial Hopf algebra of shuffles
and the indiscernability respectively, for polyzetas.
\item These will be exploited, in Section \ref{Polysystem}, to expand solutions, of nonlinear dynamical systems
with singular inputs and their ordinary and  {\it functional} differentiations.
\end{itemize}

\section{Fundation of the present framework}\label{Background}

\subsection{Background about combinatorics of shuffle and stuffle bialgebras}\label{shuffles}

\subsubsection{Sch\"utzenberger's monoidal factorization}

Let $\QX$ be equipped by the concatenation and the shuffle defined by
\begin{eqnarray*}
\forall w\in X^*,&&w\shuffle 1_{X^*}=1_{X^*}\shuffle w=w,\\
\forall x,y\in X,\forall u,v\in X^*,&&xu\shuffle yv=x(u\shuffle yv)+y(xu\shuffle v),
\end{eqnarray*}
or by their dual co-products, $\Delta_{\tt conc}$ and $\Delta_{\shuffle}$, defined by, for $w\in X^*$ and $x\in X$,
\begin{eqnarray*}
\Delta_{\tt conc}(w)=\sum_{u,v\in X^*,uv=w}u\otimes v&\mbox{and}&
\Delta_{\shuffle}(x)=x\otimes1+1\otimes x,
\end{eqnarray*}
$\Delta_{\shuffle}$ is then extended to a conc-morphism 
$\ncp{\Q}{X}\rightarrow \ncp{\Q}{X}\otimes \ncp{\Q}{X}$. 
These two  comultiplications satisfy, for any $u,v,w\in X^*$,
\begin{eqnarray*}
\scal{\Delta_{\tt conc}(w)}{u\otimes v}=\scal{w}{uv}&\mbox{and}&\scal{\Delta_{\shuffle}(w)}{u\otimes v}=\scal{w}{u\shuffle v}.
\end{eqnarray*}
One gets two mutually dual bialgebras
\begin{eqnarray*}
\calH_{\shuffle}=(\QX,{\tt conc},1_{X^*},\Delta_{\shuffle},\epsilon),&&
\calH_{\shuffle}^{\vee}=(\QX,\shuffle,1_{X^*},\Delta_{\tt conc},\epsilon).
\end{eqnarray*}

After a theorem by Radford \cite{radford}, $\Lyn X$ forms a transcendence basis of $(\QX,\shuffle,1_{X^*})$ and it can be completed then to the linear basis $\{w\}_{w\in X^*}$ which is auto-dual~:
\begin{eqnarray}\label{motsduax}
\forall v,v\in X^*,&&\scal{u}{v}=\delta_{u,v}.
\end{eqnarray}
But the elements $l\in\Lyn X-X$ are not primitive, for $\Delta_{\shuffle}$, and then $\Lyn X$ does  not constitute a  basis for $\LQX$.
Chen, Fox and Lyndon \cite{lyndon} constructed $\{P_w\}_{w\in X^*}$, so-called the Poincar\'e-Birkhoff-Witt-Lyndon basis, for ${\cal U}(\LQX)$ as follows
\begin{eqnarray}
P_x=&x&\mbox{for }x\in X,\\
P_{l}=&[P_s,P_r]&\mbox{for }l\in\Lyn X,\mbox{ standard factorization of }l=(s,r),\label{recurrence}\\
P_{w}=&P_{l_1}^{i_1}\ldots P_{l_k}^{i_k}&\mbox{for} w=l_1^{i_1}\ldots l_k^{i_k}, l_1 \succ\ldots \succ l_k, l_1\ldots,l_k  \in \Lyn X.
\end{eqnarray}
where here $\succ$ stands for the lexicographic (strict) ordering defined\footnote{In here, the order relation $\succ$ on $X^*$ is defined by, for any $u, v \in X^*$, $u \succ v$ iff $u = vw$ with $w \in X^+$ else there are $w, w_1, w_2 \in X^*$ and $a \succ b \in X$ such that $u = w a w_1$ and $v = w bw_2$.} by $x_0\prec x_1$.
Sch\"utzenberger constructed bases for $(\QX,\shuffle)$ defined by duality as follows~:
\begin{eqnarray*}
\forall u,v\in X^*,\quad\scal{S_u}{P_v}&=&\delta_{u,v}
\end{eqnarray*}
and obtained the transcendence and linear bases, $\{S_l\}_{l\in \in\Lyn X},\{S_w\}_{w\in X^*}$, as follows
\begin{eqnarray*}
S_l=&xS_u,&\mbox{for }l=xu\in\Lyn X,\cr
S_w=&\Frac{S_{l_1}^{\shuffle i_1}\shuffle\ldots\shuffle S_{l_k}^{\shuffle i_k}}{i_1!\ldots i_k!}&\mbox{for }w=l_1^{i_1}\ldots l_k^{i_k},l_1 \succ \ldots \succ l_k.
\end{eqnarray*}
After that, M\'elan\c{c}on and Reutenauer \cite{reutenauer} proved that, for any $w\in X^*$,
\begin{eqnarray}\label{basesduales}
P_w=w+\sum_{v\succ w,|v|_{X}=|w|_{X}}c_v v&\mbox{and}&S_w=w+\sum_{v\prec w,|v|_{X}=|w|_{X}}d_v v.
\end{eqnarray}
where $|w|_X=(|w|_x)_{x\in X}$ is the family of all partial degrees (number of times a letter occurs in a word).
In other words, the elements of the bases $\{S_w\}_{w\in X^*}$ and $\{P_w\}_{w\in X^*}$ are lower and upper triangular respectively and 
they are of multihomogeneous (all the monomials have the same partial degrees).

\begin{example}[of $\{P_w\}_{w\in X^*}$ and $\{S_w\}_{w\in X^*}$, \cite{hoangjacoboussous}]\label{dualbases}
\small{Let $X=\{x_0,x_1\}$ with $x_0 \prec x_1$.
$$\begin{array}{|c|c|c|}
\hline
l&P_l&S_l\\
\hline
x_0&x_0&x_0\\
x_1&x_1&x_1\\
x_0x_1&[x_0,x_1]&x_0x_1\\
x_0^2x_1&[x_0,[x_0,x_1]]&x_0^2x_1\\
x_0x_1^2&[[x_0,x_1],x_1]&x_0x_1^2\\
x_0^3x_1&[x_0,[x_0,[x_0,x_1]]]&x_0^3x_1\\
x_0^2x_1^2&[x_0,[[x_0,x_1],x_1]]&x_0^2x_1^2\\
x_0x_1^3&[[[x_0,x_1],x_1],x_1]&x_0x_1^3\\
x_0^4x_1&[x_0,[x_0,[x_0,[x_0,x_1]]]]&x_0^4x_1\\
x_0^3x_1^2&[x_0,[x_0,[[x_0,x_1],x_1]]]&x_0^3x_1^2\\
x_0^2x_1x_0x_1&[[x_0,[x_0,x_1]],[x_0,x_1]]&2x_0^3x_1^2+x_0^2x_1x_0x_1\\
x_0^2x_1^3&[x_0,[[[x_0,x_1],x_1],x_1]]&x_0^2x_1^3\\
x_0x_1x_0x_1^2&[[x_0,x_1],[[x_0,x_1],x_1]]&3x_0^2x_1^3+x_0x_1x_0x_1^2\\
x_0x_1^4&[[[[x_0,x_1],x_1],x_1],x_1]&x_0x_1^4\\
{x_0^5}x_1&[x_0,[x_0,[x_0,[x_0,[x_0,x_1]]]]]&x_0^5x_1\\\
x_0^4x_1^2&[x_0,[x_0,[x_0,[[x_0,x_1],x_1]]]]&x_0^4x_1^2\\
x_0^3x_1x_0x_1&[x_0,[[x_0,[x_0,x_1]],[x_0,x_1]]]&2x_0^4x_1^2+x_0^3x_1x_0x_1\\
x_0^3x_1^3&[x_0,[x_0,[[[x_0,x_1],x_1],x_1]]]&x_0^3x_1^3\\
x_0^2x_1x_0x_1^2&[x_0,[[x_0,x_1],[[x_0,x_1],x_1]]]&3x_0^3x_1^3+x_0^2x_1x_0x_1^2\\
x_0^2x_1^2x_0x_1&[[x_0,[[x_0,x_1],x_1]],[x_0,x_1]]&6x_0^3x_1^3+3x_0^2x_1x_0x_1^2+x_0^2x_1^2x_0x_1\\
x_0^2x_1^4&[x_0,[[[[x_0,x_1],x_1],x_1],x_1]]&x_0^2x_1^4\\
x_0x_1x_0x_1^3&[[x_0,x_1],[[[x_0,x_1],x_1],x_1]]&4x_0^2x_1^4+x_0x_1x_0x_1^3\\
x_0x_1^5&[[[[[x_0,x_1],x_1],x_1],x_1],x_1]&x_0x_1^5\\
\hline
\end{array}$$}
\end{example}

Then, Sch\"utzenberger's factorization of the diagonal series $\calD_X$ follows \cite{reutenauer}
\begin{eqnarray}\label{simple1}
\calD_X:=\sum_{w\in X^*}w\otimes w=\sum_{w\in X^*}S_w\otimes P_w\label{factorisation1}=\prod_{l\in\Lyn X}^{\searrow}\exp(S_l\otimes P_l)\label{factorisation2}.
\end{eqnarray}

\subsubsection{Extended Sch\"utzenberger's monoidal factorization}
Let us define the commutative  product over $\QY$, denoted by $\mu$, as follows
\begin{eqnarray*}
\forall y_n,y_m\in Y,&&\mu(y_n, y_m)=y_{n+m},
\end{eqnarray*}
or by its associated coproduct, $\Delta_\mu$, defined by
\begin{eqnarray*}
\forall y_n\in Y,&&\Delta_{\mu}(y_n)=\sum_{i=1}^{n-1}y_i\otimes y_{n-i}
\end{eqnarray*}
satisfying,
\begin{eqnarray*}
\forall x,y,z\in Y,&&\scal{\Delta_{\mu}}{y\otimes z}=\scal{x}{\mu(y,z)}.
\end{eqnarray*}

Let $\QY$ be equipped by
\begin{enumerate}
\item The concatenation (or by its associated coproduct, $\Delta_{\tt conc}$).
\item The {\it shuffle} product, {\it i.e.} the commutative product defined by \cite{fliess1}
\begin{eqnarray*}
\forall w\in Y^*,&&w\shuffle 1_{Y^*}=1_{Y^*}\shuffle w=w,\\
\forall x,y\in Y,u,v\in Y^*,&&xu\shuffle yv=x(u\shuffle yv)+y(xu\shuffle v)
\end{eqnarray*}
or with its associated coproduct, $\Delta_{\minishuffle}$, defined, on the letters, by
\begin{eqnarray*}
\forall y_k\in Y,&&\Delta_{\minishuffle}y_k=y_k\otimes1+1\otimes y_k
\end{eqnarray*}
and extended by morphism. It satisfies
\begin{eqnarray*}
\forall u,v,w\in Y^*,&&\langle\Delta_{\minishuffle}w\mid u\otimes v\rangle=\langle w\mid u\shuffle v\rangle.
\end{eqnarray*}

\item The {\it quasi-shuffle} product, {\it i.e.} the commutative product defined  by \cite{hoffman}
\begin{eqnarray*}
\forall w\in Y^*,&&w\stuffle 1_{Y^*}=1_{Y^*}\stuffle w=w,\\
\forall x,y\in Y,u,v\in Y^*,&&y_iu\stuffle y_jv=y_j(y_iu\stuffle v)+y_i(u\stuffle y_jv)\cr
&&+\mu(y_ i,y_j)(u\stuffle v)
\end{eqnarray*}
or with its associated coproduct, $\Delta_{\ministuffle}$,  defined, on the letters, by
\begin{eqnarray*}
\forall y_k\in Y,&&\Delta_{\ministuffle}y_k=\Delta_{\minishuffle}y_k+\Delta_{\mu}y_k
\end{eqnarray*}
and extended by morphism. It satisfies
\begin{eqnarray*}
\forall u,v,w\in Y^*,&&\langle\Delta_{\ministuffle}w\mid u\otimes v\rangle=\langle w\mid u\stuffle v\rangle.
\end{eqnarray*}
\end{enumerate}
Hence, with the counit $\tt e$ defined by, for any $P\in\QY$, ${\tt e}(P)=\scal{P}{1_{Y^*}}$, one gets two pairs of mutually dual bialgebras
\begin{eqnarray*}
\calH_{\minishuffle}=(\QY,{\tt conc},1_{Y^*},\Delta_{\shuffle},{\tt e})&\mbox{and}&
\calH_{\minishuffle}^{\vee}=(\QY,\shuffle,1_{Y^*},\Delta_{\tt conc},{\tt e}),\\
\calH_{\ministuffle}=(\QY,{\tt conc},1_{Y^*},\Delta_{\stuffle},{\tt e})&\mbox{and}&
\calH_{\ministuffle}^{\vee}=(\QY,\stuffle,1_{Y^*},\Delta_{\tt conc},{\tt e}).
\end{eqnarray*}

By the CQMM theorem (see \cite{BDHHT}), the connected $\N$-graded,
co-commutative Hopf algebra $\calH_{\minishuffle}$ is isomorphic to the enveloping algebra
of the Lie algebra of its  primitive elements which is equal to ${\calL ie}_\QY$~:
\begin{eqnarray*}
\calH_{\minishuffle}\cong\calU({\calL ie}_\QY)
&\mbox{and}&
\calH_{\minishuffle}^{\vee}\cong\calU({\calL ie}_\QY)^{\vee}.
\end{eqnarray*}
Hence, let us consider \cite{lyndon}
\begin{enumerate}
\item The PBW-Lyndon basis $\{p_w\}_{w\in Y^*}$ for ${\cal U}({\calL ie}_\QY)$ constructed recursively
$$\left\{\begin{array}{llll}
p_y&=&y& \mbox{for }y\in Y,\\
p_{l}&=&[p_s,p_r]&\mbox{for }l\in\Lyn Y,\mbox{ standard factorization of }l=(s,r),\\
p_{w}&=&p_{l_1}^{i_1}\ldots p_{l_k}^{i_k}&\mbox{for }w=l_1^{i_1}\ldots l_k^{i_k},l_1 \succ \ldots \succ l_k,l_1\ldots,l_k\in\Lyn Y,
\end{array}\right.$$
\item And, by duality\footnote{The dual family of a basis lies in the algebraic dual
which is here the space of noncommutative series, but as the enveloping algebra under consideration is graded
in finite dimensions (here by the multidegree), these series are in fact (multihomogeneous) polynomials.},
the linear basis $\{s_w\}_{w\in Y^*}$ for $(\QY,\shuffle,1_{Y^*})$, {\it i.e.}
\begin{eqnarray*}
\forall u,v\in Y^*,&&\scal{p_u}{s_v}=\delta_{u,v}.
\end{eqnarray*}
This basis can be computed recursively as follows \cite{reutenauer}
$$\left\{\begin{array}{llll}
s_y&=&y,&\mbox{for }y\in Y,\\
s_l&=&ys_u,&\mbox{for }l=yu\in\Lyn Y,\\
s_w&=&\Frac{s_{l_1}^{\shuffle i_1}\shuffle\ldots\shuffle s_{l_k}^{\shuffle i_k}}{i_1!\ldots i_k!}
&\mbox{for }w=l_1^{i_1}\ldots l_k^{i_k},l_1 \succ \ldots \succ l_k \in \Lyn Y.
\end{array}\right.$$
\end{enumerate}
As in (\ref{simple1}), one also has Sch\"utzenberger's factorization for the diagonal series $\calD_Y$
\begin{eqnarray*}
\calD_Y:=\sum_{w\in Y^*}w\otimes w=\sum_{w\in Y^*}s_w\otimes p_w=\Prod_{l\in\Lyn Y}^{\searrow}\exp(s_l\otimes p_l).
\end{eqnarray*}
Similarly, by the CQMM theorem, the connected $\N$-graded,
co-commutative Hopf algebra $\calH_{\ministuffle}$ is isomorphic to the enveloping algebra of
\begin{eqnarray*}
\mathrm{Prim}(\calH_{\ministuffle})=\mathrm{Im}(\pi_1)=\mathrm{span}_{\Q}\{\pi_1(w)\vert{w\in Y^*}\},
\end{eqnarray*}
where, for any $w\in Y^*,\pi_1(w)$ is obtained as follows \cite{BDHHT,acta}
\begin{eqnarray}\label{pi1}
\pi_1(w)=w+\sum_{k\ge2}\frac{(-1)^{k-1}}k\sum_{u_1,\ldots,u_k\in Y^+}
\scal{ w}{u_1\stuffle\ldots\stuffle u_k}\;u_1\ldots u_k.
\end{eqnarray}
Note that Equation (\ref{pi1}) is equivalent to the following identity \cite{BDHHT,acta,VJM}
\begin{eqnarray}\label{exp1}
w=\sum_{k\ge0}\frac1{k!}\sum_{u_1,\ldots,u_k\in Y^*}
\scal{ w}{u_1\stuffle\ldots\stuffle u_k}\;\pi_1(u_1)\ldots\pi_1(u_k).
\end{eqnarray}
In particular, for any $y_k\in Y$, we have successively \cite{BDHHT,acta,VJM}
\begin{eqnarray}\label{pi1bis}
\pi_1(y_k)&=&y_k+\sum_{l\ge2}\frac{(-1)^{l-1}}{l}\sum_{j_1,\ldots,j_l\ge1\atop j_1+\ldots+j_l=k}y_{j_1}\ldots y_{j_l},\\
y_n&=&\sum_{k\ge1}\frac{1}{k!}\sum_{s'_1+\cdots +s'_k=n}\pi_1(y_{s'_1})\ldots\pi_1(y_{s'_k})\label{pi2letter}
\end{eqnarray}
Hence, by introducing the new alphabet $\bar Y=\{\bar y\}_{y\in Y}=\{\pi_1(y)\}_{y\in Y}$, one has
\begin{eqnarray*}
(\Q\langle\bar Y\rangle,{\tt conc},1_{\bar Y^*},\Delta_{\minishuffle})&\cong&(\QY,{\tt conc},1_{Y^*},\Delta_{\ministuffle})
\end{eqnarray*}
as one can prove through \eqref{pi2letter} that the endomorphism $y\mapsto \bar{y}$ is, in fact, an isomorphism
\begin{eqnarray*}
\calH_{\ministuffle}\cong&\calU({\calL ie}_{\Q}\pol{\bar Y})&\cong\calU(\mathrm{Prim}(\calH_{\ministuffle})),\\
\calH_{\ministuffle}^{\vee}\cong&\calU({\calL ie}_{\Q}\pol{\bar Y})^{\vee}&\cong\calU(\mathrm{Prim}(\calH_{\ministuffle}))^{\vee}.
\end{eqnarray*}
By considering
\begin{enumerate}
\item The PBW-Lyndon basis $\{\Pi_w\}_{w\in Y^*}$ for $\calU(\mathrm{Prim}(\calH_{\ministuffle}))$ constructed recursively as follows \cite{acta}
$$\left\{\begin{array}{llll}
\Pi_y&=&\pi_1(y)& \mbox{for }y\in Y,\\
\Pi_{l}&=&[\Pi_s,\Pi_r]&\mbox{for }l\in\Lyn Y,\mbox{ standard factorization of }l=(s,r),\\
\Pi_{w}&=&\Pi_{l_1}^{i_1}\ldots\Pi_{l_k}^{i_k}&\mbox{for }w=l_1^{i_1}\ldots l_k^{i_k},l_1 \succ \ldots \succ l_k,l_1\ldots,l_k\in\Lyn Y,
\end{array}\right.$$
\item And, by duality, the linear basis $\{\Sigma_w\}_{w\in Y^*}$ for $(\QY,\stuffle,1_{Y^*})$, {\it i.e.}
\begin{eqnarray*}
\forall u,v\in Y^*,&&\scal{\Pi_u}{\Sigma_v}=\delta_{u,v}.
\end{eqnarray*}
This basis can be computed recursively as follows \cite{BDM,acta}
$$\left\{\begin{array}{ll}
\Sigma_y=y,
&\mbox{for }y\in Y,\\
\Sigma_l\;=
\Sum_{(!)}\frac{y_{s_{k_1}+\cdots+s_{k_i}}}{i!}\Sigma_{l_1\cdots l_n},
&\mbox{for }l=y_{s_1}\ldots y_{s_w}\in\Lyn Y,\\
\Sigma_w=\displaystyle\frac{\Sigma_{l_1}^{\stuffle i_1}\stuffle\ldots\stuffle\Sigma_{l_k}^{\stuffle i_k}}{i_1!\ldots i_k!},
&{\displaystyle\mbox{for }w=l_1^{i_1}\ldots l_k^{i_k},\;\mbox{with}\atop\displaystyle l_1 \succ \ldots \succ l_k \in \Lyn Y.}
\end{array}\right.$$
In $(!)$, the sum is taken over all subsequences $\{k_1,\ldots,k_i\}\allowbreak \subset \allowbreak\{1,\ldots,k\}$
and all Lyndon words $l_1\succeq \cdots \succeq l_n$ such that
$(y_{s_1},\ldots,y_{s_k})\stackrel{*}{\Leftarrow}(y_{s_{k_1}},\ldots,y_{s_{k_i}},\allowbreak l_1,\ldots,l_n)$,
where $\stackrel{*}{\Leftarrow}$ denotes the transitive closure of the relation on standard  sequences,
denoted by $\Leftarrow$  (see \cite{BDM}).
\end{enumerate}
We also proved that, for any $w\in Y^*$, \cite{BDHHT,acta,VJM}
\begin{eqnarray}\label{basesdualles}
\Pi_w=w+\sum_{v \succ w,(v)=(w)}e_v v&\mbox{and}&\Sigma_w=w+\sum_{v \prec w,(v)=(w)}f_vv.
\end{eqnarray}
In other words, the elements of the bases $\{\Sigma_w\}_{w\in Y^*}$ and $\{\Pi_w\}_{w\in Y^*}$
are lower and upper triangular respectively and they are of homogeneous in weight.

We also get the extended Sch\"utzenberger's factorization of $\calD_Y$ \cite{BDHHT,acta,VJM}
\begin{eqnarray*}
\calD_Y=\sum_{w\in Y^*}\Sigma_w\otimes\Pi_w=\Prod_{l\in\Lyn Y}^{\searrow}\exp(\Sigma_l\otimes\Pi_l).
\end{eqnarray*}

\begin{example}[of $\{\Pi_w\}_{w\in Y^*}$ and $\{\Sigma_w\}_{w\in Y^*}$, \cite{BDM}]
\small{$$\begin{array}{|c|c|c|}
\hline
l&\Pi_l&\Sigma_l\\
\hline
y_2&y_2-\frac{1}{2}y_1^2&y_2\\
y_1^2&y_1^2&\frac{1}{2}y_2+y_1^2\\
y_3&y_3-\frac{1}{2}y_1y_2-\frac{1}{2}y_2y_1+\frac{1}{3}y_1^3&y_3\\
y_2y_1&y_2y_1-y_2y_1&\frac{1}{2}y_3+y_2y_1\\
y_1y_2&y_2y_1-\frac{1}{2}y_1^3& y_1y_2\\
y_1^3&y_1^3&\frac{1}{6}y_3+\frac{1}{2}y_2y_1+\frac{1}{2}y_1y_2+y_1^3\\
y_4&y_4-\frac{1}{2}y_1y_3-\frac{1}{2}y_2^2-\frac{1}{2}y_3y_1&y_4\\
&+\frac{1}{3}y_1^2y_2+\frac{1}{3}y_1y_2y_1+\frac{1}{3}y_2y_1^2-\frac{1}{4}y_1^4&\\
y_3y_1&y_3y_1-\frac{1}{2}y_2y_1^2-y_1y_3+\frac{1}{2}y_1^2y_2&\frac{1}{2}y_4+y_3y_1\\
y_2^2&y_2^2-\frac{1}{2}y_2y_1^2-\frac{1}{2}y_1^2y_2+\frac{1}{4}y_1^4&\frac{1}{2}y_4+y_2^2\\
y_2y_1^2&y_2y_1^2-2\,y_1y_2y_1+y_1^2y_2&\frac{1}{6}y_4+\frac{1}{2}y_3y_1+\frac{1}{2}y_2^2+y_2y_1^2\\
y_1y_3&y_1y_3-\frac{1}{2}y_1^2y_2-\frac{1}{2}y_1y_2y_1+\frac{1}{3}y_1^4&y_4+y_3y_1+y_1y_3\\
y_1y_2y_1&y_1y_2y_1-y_1^2y_2&\frac{1}{2}y_4+\frac{1}{2}y_3y_1+y_2^2\\
&&+y_2y_1^2+\frac{1}{2}y_1y_3+y_1y_2y_1\\
y_1^2y_2&y_1^2y_2-\frac{1}{2}y_1^4&\frac{1}{2}y_4+y_3y_1+y_2^2+y_2y_1^2\\
&&+y_1y_3+y_1y_2y_1+y_1^2y_2\\
y_1^4&y_1^4&\frac{1}{24}y_4+\frac{1}{6}y_3y_1+\frac{1}{4}y_2^2+\frac{1}{2}y_2y_1^2\\
&&+\frac{1}{6}y_1y_3+\frac{1}{2}y_1y_2y_1+\frac{1}{2}y_1^2y_2+y_1^4\\
\hline
\end{array}$$}
\end{example}

\subsection{Indiscernability over  a class of formal power series}\label{sectionindiscernability}
\subsubsection{Residual calculus and representative series}
\begin{definition}
Let $S\in\QXX$ (resp. $\QX$) and let $P\in\QX$ (resp. $\QXX$). The left and right {\it residual}
of $S$ by $P$ are respectively the formal power series $P\resg S$ and $S\resd P$ in $\QXX$ defined by
$\pol{P\resg S\bv w}=\pol{S\bv wP}$ (resp. $\pol{S\resd P\bv w}=\pol{S\bv Pw}$).
\end{definition}

For any $S\in\QXX$ (resp. $\QX$) and $P,Q\in\QX$ (resp. $\QXX$), we straightforwardly get
$P\resg (Q\resg S)=PQ\resg S,(S\resd P)\resd Q=S\resd PQ$ and $(P\resg S)\resd Q=P\resg(S\resd Q)$.

In case $x,y\in X$ and $w\in X^* $, we get\footnote{For any words $u,v\in X^*$, if $u=v$ then $\delta_u^v=1$ else $0$.}
$x\resg (wy)=\delta_x^yw$ and $xw\resd y=\delta_x^yw$.

\begin{lemma}{\rm (Reconstruction lemma)}
Let $S\in\QXX$. Then
\begin{eqnarray*}
S=\pol{S\bv1_{X^*}}+\sum_{x\in X}x(S\resd x)=\pol{S\bv1_{X^*}}+\sum_{x\in X}(x\resg S)x.
\end{eqnarray*}
\end{lemma}

\begin{theorem}
Le $\delta\in\mathfrak{Der}(\QX,\shuffle,1_{X^*})$. Moreover, we suppose  that $\delta$ is {locally nilpotent}\footnote{
$\phi\in End(V)$ is said to be locally nilpotent iff, for any $v\in V$, there exists $N\in \N$ s.t. $\phi^N(v)=0$.}.
Then the family $(t\delta)^n/{n!}$ is summable and its sum, denoted $\exp(t\delta)$, is is a one-parameter group of automorphisms of $(\QX,\shuffle,1_{X^*})$.
\end{theorem}

\begin{theorem}
Let $L$ be a Lie series, {\it i.e.} $\Delta_{\minishuffle}(L)=L\hat\otimes1+1\hat\otimes L$.
Let $\delta^r_L,\delta^l_L$ be defined respectively by $\delta^r_L(P):=P\rg L,\delta^l_L(P):=L\rd P$.
Then $\delta^r_L,\delta^l_L$ are locally nilpotent derivations of $(\QX,\shuffle,1_{X^*})$.
Hence, $\exp(t\delta^r_L),\exp(t\delta^l_L)$ are one-parameter groups of $Aut(\QX,\shuffle,1_{X^*})$
and $\exp(t{\delta^r_L})P=P\rg\exp(tL),\exp(t{\delta^l_L})P=\exp(tL)\rd P$.
\end{theorem} 

\begin{example}\label{reslettre}
Since $x_1\triangleleft$ and $\triangleright x_0$ are derivations and the polynomials
$\{\Sigma_l\}_{l\in\in\Lyn X-X}$ belong to $x_0\mathbb{Q}\langle X\rangle x_1$ then
$x_1\triangleleft l=l\triangleright x_0=0$ and $x_1\triangleleft\check S_l=\check S_l\triangleright x_0=0$.
\end{example}

\begin{theorem}\label{representativeseries}
Let $S\in\QXX$. The following properties are equivalent:
\begin{enumerate}
\item The left $\C$-module $Res_{g}(S)=\text{span}\{w\resg S\bv w\in X^*\}$ is finite dimensional.
\item The right $\C$-module $Res_{d}(S)=\text{span}\{S\resd w\bv w\in X^*\}$ is finite dimensional.
\item There are matrices $\lambda\in\llm M_{1,n}(\mathbb{Q})$, $\eta\in\llm M_{n,1}(\mathbb{Q})$
and $\mu:X^*\longrightarrow\llm M_{n,n}$, such that
\begin{eqnarray*}
S=\sum_{w\in X^*}[\lambda\mu(w)\eta]\;w=\lambda\biggl(\prod_{l\in\Lyn X}^{\searrow}e^{\mu(S_l)\;P_l}\biggr)\eta.
\end{eqnarray*}
\end{enumerate}
\end{theorem}

A series that satisfies the items of Theorem \ref{representativeseries} will be called {\it representative series}.
This concept can be found in \cite{abe,Duchamp,hochschild,DT2}. The two first items are in
\cite{fliess0,hespel}. The third can be deduced from \cite{ChariPressley,Duchamp}
for example and it was used to factorize, for the first time, by Lyndon words,
the output of bilinear and analytical dynamical systems respectively in
\cite{IMACS0,hoangjacoboussous} and to study polylogarithms, hypergeometric functions
and associated functions in \cite{FPSAC95,FPSAC96,orlando}.
The dimension of $Res_g(S)$ is equal to that of $Res_d(S)$, and to the minimal dimension
of a representation satisfying the third point of Theorem \ref{representativeseries}.
This rank is then equal to the rank of the Hankel matrix of $S$, {\it i.e.} the infinite matrix
$(\langle S\bv uv\rangle)_{u,v\in X}$ indexed by $X^*\times X^*$ so called {\it Hankel rank}\footnote{
{\it i.e.}  the dimension of $\mathrm{span}\{S\resd\Pi\bv\Pi\in\CX\}$
(resp. $\mathrm{span}\{\Pi\resg S\bv\Pi\in\CX\}$).} of $S$ \cite{fliess0,hespel}.
The triplet $(\lambda,\mu,\eta)$ is called a {\it linear representation} of $S$\footnote{
The minimal representation of $S$ as being a representation of $S$ of minimal dimension.
It can be shown that all minimal representations are isomorphic (see \cite{berstel}).} .
$S$ is called {\it rational} if it belongs to the closure by +, conc and star operation of proper elements\footnote{
For any  {\it proper} series $S$, {\it i.e.} $\scal{S}{1_{X^*}}=0$, the series $S^*=1+S+S^2+\ldots$ is called ``star of $S$''}.
Any noncommutative power series is representative if and only if it is rational \cite{berstel,schutz}.  These rationality properties can be expressed 
in terms of differential operators in noncommutative geometry \cite{Duchamp}.

\subsubsection{Background on continuity and indiscernability}

\begin{definition}{\rm (\cite{these,cade})}\label{indiscernability}
Let $\calH$ be a class of $\CXX$ and $S\in\CXX$.
\begin{enumerate}
\item $S$ is said to be {\it continuous} over $\calH$ if for any $\Phi\in\calH$,
the following sum, denoted by $\langle S\bbv\Phi\rangle$, is absolutely convergent $\sum_{w\in X^*}\langle S\bv w\rangle\langle\Phi\bv w\rangle$.

The set of continuous power series over $\calH$ will be denoted by $\CcXX$.

\item $S$ is said to be {\it indiscernable} over $\calH$ if and only if, for any $\Phi\in\calH$, $\langle S\bbv\Phi\rangle=0$.
\end{enumerate}
\end{definition}

\begin{proposition}\label{nulle}
Let $S\in\CcXX$. $\calH$ is a monoid containing $X$ and $\{e^{tx}\}_{x\in X}^{t\in\C}$.
\begin{enumerate}
\item If $S$ is indiscernable over $\calH$ then for any $x\in X$,
$x\triangleleft S$ and $S\triangleright x$ belong to $\CcXX$ and they are indiscernable over $\calH$.

\item $S$ is indiscernable over $\calH$ if and only if $S=0$.
\end{enumerate}
\end{proposition}

\begin{proof}
\begin{enumerate}
\item Of course, $x\triangleleft S$ and $S\triangleright x$ belong to $\CcXX$. Let us calculate
$\langle x\triangleleft S\bbv\Phi\rangle=\langle S\bbv\Phi x\rangle$
and $\langle S\triangleright x\bbv\Phi\rangle =\langle S\bbv x\Phi\rangle$.
Since $S$ is indiscernable over $\calH$ and note that $x\Phi, \Phi x \in \calH$ for evvery $x\in X; \Phi \in \calH$, then
\begin{eqnarray*}
\langle S\bbv\Phi x\rangle 
=0,&\mbox{and}&\langle S\bbv x\Phi\rangle  =0.
\end{eqnarray*}
Hence $x\triangleleft S$ and $S\triangleright x$ belong to $\CcXX$ are indiscernable over $\calH$.
\item $S=0$ is indiscernable over $\calH$.
Conversely, if $S$ is indiscernable over $\calH$ then by the previous point and by
induction on the length of $w\in X^*$, $w\triangleleft S$ is indiscernable over $\calH$. In particular,
$\scal{ w\triangleleft S}{\mathrm{Id}_{\calH}}=\scal{S}{w}=0$. In other words, $S=0$.
\end{enumerate}
\end{proof}

\subsection{Polylogarithms and harmonic sums}

\subsubsection{Structure of polylogarithms and of harmonic sums}\label{Structures}

Let $\Omega:={\mathbb C}-(]-\infty,0]\cup[1,+\infty[)$ and let $\calC:=\C[z,1/z,1/{1-z}]$.
Note that the neutral element of $\calC$, for the pointwise product,
is $1_{\Omega}:\Omega\longrightarrow\C$ such that $z\longmapsto1$.

One can check that $\Li_{s_1,\ldots,s_r}$ is obtained as the iterated integral over the differential forms
$\omega_0(z)={dz}/z$ and $\omega_1(z)={dz}/(1-z)$ and along the path $0\path z$ \cite{IMACS}~:
\begin{eqnarray}
\Li_{s_1,\ldots,s_r}(z)=\alpha_0^z(x_0^{s_1-1}x_1\ldots x_0^{s_r-1}x_1)
=\sum_{n_1>\ldots>n_r>0}\frac{z^{n_1}}{n_1^{s_1}\ldots n_r^{s_r}}.
\end{eqnarray}
By (\ref{correspondence}),  $\Li_{s_1,\ldots,s_r}$ is then denoted also by $\Li_{x_0^{s_1-1}x_1\ldots x_0^{s_r-1}x_1}$
or $\Li_{y_{s_1}\ldots y_{s_r-1}}$ \cite{FPSAC95,FPSAC96,SLC43}.

\begin{example}[of $\Li_2=\Li_{x_0x_1}$]
\small{\begin{eqnarray*}
\alpha_0^z(x_0x_1)=\int_0^z\frac{ds}s\int_0^s\frac{dt}{1-t}=\int_0^z\frac{ds}s\int_0^s{dt}\sum_{k\ge0}t^k
=\sum_{k\ge1}\int_0^z{ds}\frac{s^{k-1}}k=\sum_{k\ge1}\frac{z^k}{k^2}.
\end{eqnarray*}}
\end{example}

The definition of polylogarithms is extended over the words $w\in X^*$ by putting $\Li_{x_0}(z):=\log(z)$.
The $\{\Li_w\}_{w\in X^*}$ are $\calC$-linearly independent \cite{DDMS,FPSAC98,SLC43} and then the following function,
for $v=y_{s_1}\ldots y_{s_r}\in Y^*$, are also $\C$-linearly independent \cite{DDMS,words03}
\begin{eqnarray*}
\P_v(z):=\frac{\Li_v(z)}{1-z}=\sum_{N\ge0}\H_v(N)\;z^N,&\mbox{where}&
\H_v(N):=\sum_{N\ge n_1>\ldots>n_r>0}\frac1{n_1^{s_1}\ldots n_r^{s_r}}.
\end{eqnarray*}

\begin{proposition}{\rm (\cite{words03})}\label{isomorphisms}
By linearity, the following maps are isomorphisms of algebras
\begin{eqnarray*}
\P_{\bullet}:(\C\pol{Y},\stuffle)\longrightarrow\left(\C\{\P_w\}_{w\in Y^*},\odot\right),&&u\longmapsto\P_u,\\
\H_{\bullet}:(\C\pol{Y},\stuffle)\longrightarrow\left(\C\{\H_w\}_{w\in Y^*},.\right),&&u\longmapsto\H_u=\{\H_u(N)\}_{N \geq 0}.
\end{eqnarray*}
\end{proposition}

\begin{theorem}{\rm (\cite{cade})}
The Hadamard $\calC$-algebra of $\{\P_w\}_{w\in Y^*}$ can be identified with that of $\{\P_l\}_{l\in\Lyn Y}$.
In the same way, the algebra of harmonic sums $\{\H_w\}_{w\in Y^*}$
with polynomial coefficients can be identified with that of $\{\H_l\}_{l\in\Lyn Y}$.
\end{theorem}

Let $\L,\P$ and $\H$ be the noncommutative generating series of respectively $\{\Li_w\}_{w\in X^*},
\allowbreak\{\P_w\}_{w\in X^*}$ and $\{\H_w(N)\}_{w\in Y^*}$, for $\abs{z}<1$ and $N>1$ \cite{FPSAC98,words03}~:
\begin{eqnarray}\label{ngs}
\L(z)=\sum_{w\in X^*}\Li_w(z)w; \quad &\P(z)=\Frac{\L(z)}{1-z}; \quad &\H(N)=\sum_{w\in Y^*}\H_w(N)\;w.
\end{eqnarray}

\begin{definition}[Polylogarithms and harmonic sums at negative multi-indices]\label{polylogneg}
For any $s_1,\ldots,s_r\in(\N)^r$, let us define \cite{GHN}, for $\abs{z}<1$ and $N>0$,
\begin{eqnarray*}
\Li_{-s_1,\ldots, -s_r}(z):=\sum_{n_1>\ldots>n_r>0}n_1^{s_1}\ldots n_r^{s_r}\;z^{n_1}&\mbox{and}&
\H_{-s_1,\ldots,-s_r}(N):=\sum_{N\ge n_1>\ldots>n_r>0}n_1^{s_1}\ldots n_r^{s_r}.
\end{eqnarray*}
The ordinary generating series, $\P_{-s_1,\ldots,-s_r}(z)$, of $\{\H_{-s_1,\ldots,-s_r}(N)\}_{N\ge0}$ is
\begin{eqnarray*}
\P_{-s_1,\ldots, -s_r}(z):=\sum_{N\ge0}\H_{-s_1,\ldots,-s_r}(N)\;z^N=\frac{1}{1-z}\Li_{-s_1,\ldots,-s_r}(z).
\end{eqnarray*}
\end{definition}

Now, let\footnote{with $y_0 \succ y_1$.} $Y_0=Y\cup\{y_0\}$ and let $Y_0^*$ denotes the free monoid generated by $Y_0$ admitting $1_{Y_0^*}$  as neutral element.
As in (\ref{correspondence}), let us introduce another correspondence
\begin{eqnarray*}
(s_1,\ldots,s_r)\in\N^r&\leftrightarrow&y_{s_1}\ldots y_{s_r}\in Y_0^*.
\end{eqnarray*}
In all the sequel, for some convenience, we will also adopt the following notations, for any $w=y_{s_1}\ldots y_{s_r}\in Y_0^*$,
$$\Li^-_w=\Li_{-s_1,\ldots,-s_r}; \quad \P^-_w=\P_{-s_1,\ldots,-s_r} \quad \mbox{and} \quad \H^-_w=\H_{-s_1,\ldots,-s_r}.$$

\begin{example}[$\Li^-_{y_0^r}$ and $\H^-_{y_0^r}$]\label{Li0}
\small{By Proposition (\ref{noyaux1}), we have $\Li^-_{y_0^r}=\lambda^r$. Hence,
\begin{eqnarray*}
\frac{\Li^-_{y_0^r}(z)}{1-z}=\frac{z^r}{(1-z)^{r+1}}=\sum_{N\ge0}{N\choose r}z^N
&\mbox{and then}&\H^-_{y_0^r}(N)={N\choose r}.
\end{eqnarray*}}
\end{example}

\begin{definition}
With the convention $\H^-_{1_{Y^*_0}}= 1$, we put
\begin{eqnarray*}
\L^-(z):=\sum_{w\in Y_0^*}\Li^-_{w}(z)w; \quad &\P^{-}(z):=\Frac{\L^{-}(z)}{1-z}; \quad &\H^{-}(N):=\sum_{w\in Y_0^*}\H^-_{w}(N) w.
\end{eqnarray*}
\end{definition}

Since, for $y_k\in Y,u\in Y^*$ (resp. $y_k\in Y_0,u \in Y_0^*$) and $N\ge1$, one has
$\H_{y_ku}(N)-\H_{y_ku}(N-1)=N^{-k}\H_u(N-1)$ (resp. $\H^-_{y_ku}(N)-\H^-_{y_ku}(N-1)=N^k\H^-_u(N-1)$). Then

\begin{proposition}
$\H$ and $\H^-$ satisfy the following difference equations
\begin{eqnarray*}
\H(N)&=&\biggl(1_{Y^*}+\Sum_{k\ge1}\frac{y_k}{N^k}\biggr)\H(N-1)=\Prod_{n=1}^{N}\biggl(1_{Y^*}+\sum_{k\ge1}\frac{y_k}{n^k}\biggr)= 1_{Y^*} + \sum_{w\in Y^*,\abs{w}\ge N}\H_w(N)\;w ,\\
\H^-(N)&=&\biggl(1_{Y_0^*}+\Sum_{k\ge0}{y_k}{N^k}\biggr)\H^-(N-1)=\Prod_{n=1}^{N}\biggl(1_{Y_0^*}+\sum_{k\ge0}{y_k}{n^k}\biggr)= 1_{Y_0^*} + \sum_{w\in Y_0^*,\abs{w}\ge N}\H^-_w(N)\;w.
\end{eqnarray*}
Hence, for any $w\in Y^*$ (resp. $w\in Y_0^*$),  $\H_w(N)$ (resp. $\H^-_w(N)$) is of valuation $N$.
\end{proposition}

In all the sequel,  the {\it length} and the  {\it weight} of $u=y_{i_1}\ldots y_{i_k}\in Y^*$ are defined respectively as the numbers $\abs{u}=k$ and $(u)=i_1+\ldots+ i_k$.

\begin{definition}\label{gandh}
Let $g,h\in\Q \left\langle{\left\langle{Y_0}\right\rangle}\right\rangle[\![t]\!]$ be defined as follows (here, $\abs{1_{Y_0^*}}=(1_{Y_0^*})=0$)
\begin{eqnarray*}
h(t):=\sum_{w\in Y_0^*}{((w)+\abs w)!}{t^{(w)+\abs w}}w
&\mbox{and}&
g(t):=\sum_{w\in Y_0^*}t^{(w)+\abs w}w=\biggl(\Sum_{y\in Y_0}t^{(y)+1}y\biggr)^*.
\end{eqnarray*}
\end{definition}

\begin{remark}
\begin{enumerate}
\item The generating series $h$ is an extension of the Euler series $\sum_{n\ge0}n!t^n$ and it can be obtained as Borel-Laplace transform of $g$.
\item The ordinary generating series ${\cal Y}(t):=1+\sum_{r\ge 0}y_r\;t^r$
and its inverse are group-like. The generating series $\Lambda(t)=\sum_{w \in Y^*_0} t^{(w) + |w|}w$ can be obtained from
$1/{\cal Y}(t)$ by use the following change of alphabet $y_r\leftarrow ty_r$ it can be expressed as
\begin{eqnarray*}
g(t)=\biggl(1-\Sum_{r\ge0}(-ty_r)\;t^r\biggr)^{-1}=\biggl(\Sum_{r\ge0}(-ty_r)\;t^r\biggr)^*.
\end{eqnarray*}
\end{enumerate}
\end{remark}

Now, let us consider the following differential and integration operators
acting on $\C\{\Li_w\}_{w\in X^*}$ which can be extended over $\calC\{\Li_w\}_{w\in X^*}$ \cite{{orlando}}~:
\begin{eqnarray*}
\partial_z={d}/{dz},\ \theta_0=z{d}/{dz},\ \theta_1=(1-z){d}/{dz},
\iota_0:\Li_w\longmapsto\Li_{x_0w},\ \iota_1:\Li_w\longmapsto\Li_{x_1w}
\end{eqnarray*}

Let $\Theta$ and $\Im$ be monoid morphisms such that $\Theta(1_{X^*})=\Im(1_{X^*})=\mathrm{Id}$ and,
for $x_i\in X,v\in X^*$, $\Theta(vx_i)=\Theta(v)\theta_i$ and $\Im(vx_i)=\Im(v)\iota_i$.
By extension, we obtain
$\calH_{\tt conc}\cong(\Q\langle\Theta(X)\rangle,{\tt conc},\mathrm{Id},\Delta_{\shuffle},\epsilon)$
and $\calH_{\shuffle}\cong(\Q\langle\Im(X)\rangle,\shuffle,\mathrm{Id},\Delta_{\tt conc},\epsilon)$.
Hence,

\begin{proposition}
\begin{enumerate}
\item The operators $\{\theta_0,\theta_1,\iota_0,\iota_1\}$ satisfy in particular,
$$\begin{array}{rcl}
\theta_1+\theta_0=\bigl[\theta_1,\theta_0\bigr]=\partial_z&\mbox{and}&\forall k=0,1,\theta_k\iota_k=\mathrm{Id},\cr
[\theta_0\iota_1,\theta_1\iota_0]=0&\mbox{and}&(\theta_0\iota_1)(\theta_1\iota_0)=(\theta_1\iota_0)(\theta_0\iota_1)=\mathrm{Id}.
\end{array}$$

\item For any $w=y_{s_1}\ldots y_{s_r}\in Y^*$ ($\pi_X(w)=x_0^{s_1-1}x_1\ldots x_0^{s_r-1}x_1$)
and $u=y_{t_1}\ldots y_{t_r}\in Y_0^*$, we can rephrase $\Li_w,\Li^-_u$ as follows
$$\begin{array}{rcl}
\Li_w=({\iota_0^{s_1-1}\iota_1\ldots\iota_0^{s_r-1}\iota_1})1_{\Omega}
&\mbox{and}&
\Li^-_u=({\theta_0^{t_1+1}\iota_1\ldots\theta_0^{t_r+1}\iota_1})1_{\Omega},\\
\theta_0\Li_{x_0\pi_X(w)}=\Li_{\pi_X(w)}&\mbox{and}&\theta_1\Li_{x_1\pi_X(w)}=\Li_{\pi_X(w)},\\
\iota_0\Li_{\pi_X(w)}=\Li_{x_0\pi_X(w)}&\mbox{and}&\iota_1\Li_{w}=\Li_{x_1\pi_X(w)}.
\end{array}$$

\item $\calC\{\Li_w\}_{w\in X^*}\cong\calC\otimes\C\{\Li_w\}_{w\in X^*}$ is closed under of $\iota_0,\iota_1,\theta_0,\theta_1$.

\item Let ${\lambda(z):={z}/{(1-z)}}\in\calC$. Then $\lambda$ and $1/{\lambda}$ are  the eigenvalues of
$\theta_0\iota_1$ and $\theta_1\iota_0$ within $\calC\{\Li_w\}_{w\in X^*}$ respectively~:
$$\begin{array}{rcllcc}
\forall f\in\calC\{\Li_w\}_{w\in X^*},(\theta_0\iota_1)f&=&\lambda f&\mbox{and}&(\theta_1\iota_0)f=f/{\lambda}.
\end{array}$$
\item For any $n\ge0$ and $w\in X^*$, one has\footnote{For any $w = x_{i_1}\ldots x_{i_r} \in X^*$, we denote $\widetilde w = x_{i_r}\ldots x_{i_1}$.}
\begin{eqnarray*}
\Theta(\widetilde w)\Li_w=1_{\Omega}&\mbox{and}&\partial_z^n=\sum_{w\in X^n}(\Theta\otimes\Theta)\Delta_{\shuffle}(w).
\end{eqnarray*}
\item For any $P,Q\in\QX$ and $R\in\LQX$, one has
\begin{eqnarray*}
\Theta(R)\Li_{P\shuffle Q}=\Li_{(P\shuffle Q)\rg R}=(\Theta(R)\Li_P)\Li_Q+\Li_P(\Theta(R)\Li_Q).
\end{eqnarray*}
\end{enumerate}
\end{proposition}
\begin{proof}
The proofs are immediate.
\end{proof}
\begin{proposition}\cite{GHN}\label{noyaux1}
\begin{enumerate}
\item For any $w\in Y_0^*$, one has $\Li^-_w(z)=\lambda^{\abs w}(z){A^-_w(z)}{(1-z)^{-(w)}}$,
where $A^-_w$ is the extended Eulerian polynomial defined recursively as follows
\begin{eqnarray*}
A^-_w(z)=\left\{\begin{array}{rcl}
\Sum_{k=0}^{n-1}A_{n, k} z^k&\mbox{if}&w=y_k\in Y_0,\\
\Sum_{i=0}^{s_1} \binom{s_1}{i}A_{y_i}A^-_{y_{(s_1 +s_2 - i)}y_{s_3}...y_{s_r}}&\mbox{if}&w=y_ku\in Y_0Y_0^*,
\end{array}\right.
\end{eqnarray*}
and $A_{n,k}$ are Eulerian numbers satisfying $A_{n,k}=\sum_{j=0}^k(-1)^{j}{n+1\choose j}(k+1-j)^n$.

\item For any $w\in Y^*$, let us define $\{G^-_w(n)\}_{n\in\N}$ by the following generating series
\begin{eqnarray*}
\Sum_{n\ge\abs{w}}\frac{(n+1)!}{(n-\abs{w})!}G^-_w(n)z^n=\Frac{\Li^-_w(z)}{1-z}.
\end{eqnarray*}
Then  $\H^-_w(N)=(N+1)N(N-1)\ldots(N-\abs{w}+1)G^-_w(N)$.

\item $\Li^-_w(z)\in\Q[(1-z)^{-1}]\subsetneq\calC$ and $\H^-_w(N)\in\Q[N]$ of degree $\abs{w}+(w)$.
\end{enumerate} 
\end{proposition}

\begin{example}\cite{GHN}[Case of $r=1$ by Maple]\label{Euleriannumbers}
\small{\begin{enumerate}
\item Since ${A_n(z)}/{(1-z)^{n+1}}=\sum_{j\ge0}z^j(j+1)^n$ then $\Li^-_{y_n}(z)=zA_n(z)/(1-z)^{n+1}$ (see \cite{foata} for example).
For example,
$$\begin{array}{lcll}
\Li^-_{y_1}(z)=&{z}(1-z)^{-2}&=-(1-z)^{-1}+(1-z)^{-2}.\\
\Li^-_{y_2}(z)=&{z(z+1)}(1-z)^{-3}&=(1-z)^{-1}- {3}(1-z)^{-2}+{2}(1-z)^{-3}.\\
\Li^-_{y_3}(z)=&{z(z^2+4z+1)}(1-z)^{-4}&=-(1-z)^{-1}+{7}(1-z)^{-2}-{12}(1-z)^{-3}+{6}(1-z)^{-4}.
\end{array}$$
\item For any positive integer $m$, one has
\begin{eqnarray*}
\H^-_{y_m} (N)=\frac{1}{m+1}\sum_{k=0}^m\binom{m+1}{k}B_{k}(N+1)^{m+1-k}
=\frac{1}{m+1}\sum_{k = 1}^{m+1} \left[ \sum_{l=0}^{m+ 1-k}\binom{m+1}{l}\binom{m+1-l}{k}B_{l}\right] N^l,
\end{eqnarray*}
where $B_{k}$ is the $k$-th Bernoulli's number  given by its exponential generating series
\begin{eqnarray*}
\frac t{e^t-1} =\sum_{k\ge0}B_k\frac{t^{k}}{k!}.
\end{eqnarray*}
For example, (recall that $B_0=1,B_1=-1/2,B_2=1/6,B_3=0,B_4=-1/30$),
\begin{eqnarray*}
\H^-_{y_1}(N)=&(N+1)^2/2-(N+1)/2&=N(N+1)/2,\cr
\H^-_{y_2}(N)=&(N+1)^3/3 - (N+1)^2/2+ (N+1)/6&= N(2N+1)(N+1)/6,\cr
\H^-_{y_3}(N)=&(N+1)^4/4-(N+1)^3/2+(N+1)^2/4&=({N(N+1)}/2)^2.
\end{eqnarray*}
\end{enumerate}}
\end{example}

\begin{example}\cite{GHN}[Case of $r=2$ by Maple]
\small{\begin{enumerate}
\item From what precedes, $\Li^-_{y_m y_n}=(\theta_0^{m+1}\iota_1)\Li^-_{y_n}=\theta_0^m(\theta_0\iota_1)\Li^-_{y_n}$.
Since, by Example \ref{Li0}, we have $(\theta_0\iota_1)\Li^-_{y_n}=\Li^-_{y_0}\Li^-_{y_n}$ then 
$\Li^-_{y_m y_n}=\theta_0^m[\Li^-_{y_0}\Li^-_{y_n}]=\sum_{l=0}^m{m\choose l}\Li^-_{y_l}\Li^-_{y_{m+n-l}}$.
For example, 
\begin{eqnarray*}
\Li^-_{y_1^2} (z)
&=&\Li^-_{y_0}(z)\Li^-_{y_2}(z)+(\Li^{-}_{y_1}(z))^2\cr
&=&-(1-z)^{-1}+{5}(1-z)^{-2} - {7}(1-z)^{-3} + {3}(1-z)^{-4}\cr
\Li^-_{y_2y_1} (z)
&=&\Li^-_{y_0}(z)\Li^-_{y_3}(z) +3\Li^-_{y_1}(z)\Li^-_{y_2}(z)\cr
&=&(1-z)^{-1}-{11}(1-z)^{-2}+{31}(1-z)^{-3} -{33}(1-z)^{-4} + {12}(1-z)^{-5},\cr
\Li^-_{y_1y_2} (z)&=&\Li^-_{y_0}(z)\Li^-_{y_3}(z)+\Li^-_{y_1}(z)\Li^-_{y_2}(z)\cr
&=&(1-z)^{-1}-{9}(1-z)^{-2}+{23}(1-z)^{-3}-{23}(1-z)^{-4}+{8}(1-z)^{-5}.
\end{eqnarray*}
\item For any positive integers $m,n$, one has
\begin{eqnarray*}
\H^-_{y_m y_n}(N)&=&\sum^{n}_{k_1=0}\sum^{m+n+1-k_1}_{k_2 =0}\sum^{m+n+2-k_1-k_2}_{k_3=0}
\frac{B_{k_1}B_{k_2}}{(n+1)(m+n+2-k_1)}\\
&&\binom{n+1}{k_1}\binom{m+n+2-k_1}{k_2}\binom{m+n+2-k_1-k_2}{k_3}N^{k_3}.
\end{eqnarray*}
For example, 
\begin{eqnarray*}
\H^-_{y_2 y_1}(N)&=&{N(N^2-1)(12N^2 +15N +2)}/{120},\cr
\H^-_{y_2^2}(N)&=&{N(N-1)(2N+1)(2N-1)(5N+6)(N+1)}/{360},\cr
\H^-_{y_2y_3}(N)&=&{N(N-1)(N+1)(30N^4+35N^3-33N^2-35N+2)}/{840},\cr
\H^-_{y_2y_4}(N)&=&N(N-1)(N+1)(63N^5+72N^4-133N^3-138N^2+49N+30)/{2520},\cr
\H^-_{y_2y_5}(N)&=&N(N-1)(N+1)(280N^6+315N^5-920N^4-945N^3+802N^2+630N-108)/{15120},\cr
\H^-_{y_3^3}(N)&=&N(N-1)(N+1)(21N^5+36N^4-21N^3-48N^2+8)/{672}.
\end{eqnarray*}
\end{enumerate}}
\end{example}

\begin{example}\cite{GHN}[General case]
\small{\begin{enumerate}
\item One has, for any $y_{s_1}u = y_{s_1}\ldots y_{s_r} \in Y^*_0$,
\begin{eqnarray*}
\Li^-_{y_{s_1}u}&=&\theta_0^{s_1}(\theta_0\iota_1)\Li^-_u=\theta_0^{s_1}(\lambda\Li^-_u)=
\sum_{k_1=0}^{s_1}\binom{s_1}{k_1}(\theta_0^{k_1}\lambda)(\theta_0^{s_1-k_1}\Li^-_u),\\
\Li^-_{y_{s_1}\ldots y_{s_r}}
&=&\sum_{k_1=0}^{s_1}\sum_{k_2=0}^{s_1+s_2-k_1}\ldots\sum_{k_r=0}^{(s_1+\ldots+s_r)-\atop(k_1+\ldots+k_{r-1})}
\binom{s_1}{k_1}\binom{s_1+s_2-k_1}{k_2}\ldots\\
&&\binom{s_1+\ldots+s_r-k_1-\ldots-k_{r-1}}{k_r}(\theta_0^{k_r}\lambda)(\theta_0^{k_2}\lambda)\ldots(\theta_0^{k_r}\lambda).
\end{eqnarray*}
Denoting $S_2({k_i},j)$ Stirling numbers of the second kind, one has
\begin{eqnarray*}
\forall i=1,..,r,&&
\theta_0^{k_i}\lambda(z)=\left\{
\begin{array}{rcl}
\lambda(z),&\mbox{if}&{k_i}=0,\\
\Frac1{1-z}\Sum_{j=1}^{k_i}S_2({k_i},j)j!{\lambda^j(z)},&\mbox{if}&k_i>0.
\end{array}\right.
\end{eqnarray*}
In particular, if $\w\in Y^*$ then $(1-z)^{\abs{w}}\Li^-_w(z)$ is polynomial of degree $(w)$ in $\lambda(z)$.

\item We define, firstly, the {\it polynomials} $\{B_{y_{n_1} \ldots y_{n_r}}(z)\}_{n_1,\ldots , n_r \in \N}$
by their commutative exponential generating series as follows, for $z \in \C$,
\begin{eqnarray*}
\sum_{n_1,\ldots,n_r \in \N}B_{y_{n_1}\ldots y_{n_r}}(z)\frac{t_1^{n_1}\ldots t_r^{n_r}}{n_1! \ldots n_r !}
={t_1\ldots t_re^{z(t_1+ \ldots + t_r)}}{\prod_{k=1}^r(e^{t_k+\ldots +t_r} -1)^{-1}},
\end{eqnarray*}
or by the difference equation, for $n_1\in\N_+$,
\begin{eqnarray*}
B_{y_{n_1}\ldots y_{n_r}} (z+1)=B_{y_{n_1}\ldots y_{n_r}} (z)+n_1z^{n_1-1}B_{y_{n_2} \ldots y_{ n_r}}(z).
\end{eqnarray*}

For any $w\in y_sY_0^*,s>1$, we have $B_w(1)=B_w(0)$. Then let also, for any $1\leq k\leq r$,
\begin{eqnarray*}
b_w:=B_w(0)&\mbox{and}&\beta_w(z):=B_w(z)-b_w.\\
b'_{y_k}:=b_{y_k}&\mbox{and}&
b'_{y_{n_k}\ldots y_{n_r}}:=b_{y_{n_k}\ldots y_{n_r}}-\Sum_{j=0}^{r-1-k}b_{y_{n_{k+j+1}}\ldots y_{n_r}}b'_{y_{n_k}\ldots y_{n_{k+j}}}
\end{eqnarray*}
Then we have the extended Faulhaber's identities
\begin{eqnarray*}\label{Faulhaber}
\beta_{y_{n_1}\ldots y_{n_r}} (N)&=&\sum_{k=1}^{r} (\prod_{i=1}^{k}n_i)b_{y_{n_{k+1}}\ldots y_{n_r}}\H^-_{y_{n_1-1}\ldots y_{n_k-1}}(N-1),\\
\H^-_{y_{n_1} \ldots y_{n_r}} (N)&=&\frac{\beta_{y_{n_1+1} \ldots y_{n_r +1}}(N+1)
-\sum_{k=1}^{r-1}b'_{y_{n_{k+1}+1} \ldots y_{n_r +1}}\beta_{y_{n_1+1}\ldots y_{n_k +1}}(N+1)}{\prod_{i=1}^r (n_i +1)}.
\end{eqnarray*}
\end{enumerate}}
\end{example}





\begin{proposition}\cite{GHN}\label{isomorphisms}
The following maps are morphisms of algebras
\begin{eqnarray*}
\H^-:(\C\langle{Y_0}\rangle,\stuffle)\longrightarrow(\C\{\H^-_w\}_{w\in Y_0^*},.)&\mbox{and}&
\P^-:(\C\langle{Y_0}\rangle,\stuffle)\longrightarrow(\C\{\P^-w\}_{w\in Y_0^*},\odot).
\end{eqnarray*}
\end{proposition}

\begin{proof}
Recall that the quasi-symmetric functions on the variables ${\bf t}=\{t_i\}_{N\ge i\ge1}$, {\it i.e.}
\begin{eqnarray*}
\F_{s_1,\ldots,s_r}({\bf t})=\F_{y_{s_1}\ldots y_{s_r}}({\bf t})=\sum_{n_1>\ldots>n_r>0}t_{n_1}^{s_1}\ldots t_{n_r}^{s_r}
\end{eqnarray*}
satisfy the quasi-shuffle relation \cite{reutenauer}, {\it i.e.} for any $u,v\in Y_0^*$,
$\F_{u\stuffle v}({\bf t})=\F_u({\bf t})\F_v({\bf t})$.

Since $\H^-_{s_1,\ldots,s_r}(N)$ can be obtained by specializing, in $\F_{s_1,\ldots,s_r}({\bf t})$, the variables ${\bf t}$ at
\begin{eqnarray*}
\forall 1\le i\le N,t_i=i&\mbox{and}&\forall i>N,t_i=0
\end{eqnarray*}
then $\H^-$ is a morphism of algebras. Therefore, $\P^-$ is also a morphism of algebras.
\end{proof}

\subsubsection{Global renormalizations via noncommutative generating series}\label{ngs}

By (\ref{ngs}), $\L$ and $\H$ are images,  by the tensor products $\Li\otimes\mathrm{Id}$
and $\H\otimes\mathrm{Id}$, of the diagonal series $\calD_X$ and $\calD_Y$ respectively.
Then we get

\begin{theorem}[Factorization of $\L$ and of $\H$, \cite{SLC43,FPSAC98,acta}]\label{factorization}
Let
\begin{eqnarray*}
\L_{\reg}=\Prod_{l\in\Lyn X-X}^{\searrow}e^{\Li_{S_l}P_l}&\mbox{and}&
\H_{\reg}(N)=\Prod_{l\in\Lyn Y-\{y_1\}}^{\searrow}e^{\H_{\check\Sigma_l}(N)\;\Sigma_l}.
\end{eqnarray*}
Then $\L(z)=e^{-x_1\log(1-z)}\L_{\reg}(z)e^{x_0\log z}$ and $\H(N)=e^{\H_{y_1}(N)\;y_1}\H_{\reg}(N)$.
\end{theorem}

For any $l\in\Lyn X-X$ (resp. $\Lyn Y-\{y_1\}$), the polynomial $S_l$ (resp. $\Sigma_l$)
is a finite combination of words in $x_0X^*x_1$ (resp. $Y^*-y_1Y^*$). Then we can state

\begin{proposition}[\cite{acta}]\label{psikz}
Let $Z_{\minishuffle}:=\L_{\reg}(1)$ and $Z_{\ministuffle}:=\H_{\reg}(\infty)$.
Then $Z_{\minishuffle}$ and $Z_{\ministuffle}$ are group-like, for $\Delta_{\minishuffle}$ and $\Delta_{\ministuffle}$ respectively.
\end{proposition}

\begin{proposition}[Successive integrations and differentiations of $\L$, \cite{orlando}]\label{lem:derivL}
We have, for any $n\in\N$,
\begin{enumerate}
\item $\iota_0^n\L=x_0^n\resd\L$ and $\iota_1^n\L=x_1^n\resd\L$.
\item $\partial_z^n\L=D_n\L$ and $\theta_0^n\L=E_n\L$,
where\footnote{Since $\theta_0+\theta_1=\partial_z$ then we also have $\theta_1^n\L(z)=[D_n(z)-E_n(z)]\L(z)$.
The more general actions of $\{\Theta(w)\}_{w\in X^*}$ on $\L$ are more complicated to be expressed here.}
the polynomials $D_n$ and $E_n$ in $\calC\langle X\rangle$ are
\begin{eqnarray*}
D_n&=&\sum_{\wgt({\bf r})=n}\sum_{w\in X^{\deg({\bf r})}}\prod_{i=1}^{\deg({\bf r})}
\binom{\sum_{j=1}^ir_i+j-1}{r_i}\tau_{\bf r}(w),\\
E_n&=&\sum_{\wgt({\bf r})=n}\sum_{w\in X^{\deg({\bf r})}}\prod_{i=1}^{\deg({\bf r})}
\binom{\sum_{j=1}^ir_i+j-1}{r_i}\rho_{\bf r}(w),
\end{eqnarray*}
and for any $w=x_{i_1}\cdots x_{i_k}$ and ${\bf r}=(r_1,\ldots,r_k)$ of degree $\deg({\bf r})=k$
and of weight $\wgt({\bf r})=k+r_1+\cdots +r_k$, the polynomials $\tau_{\bf r}(w)=\tau_{r_1}(x_{i_1})\cdots\tau_{r_k}(x_{i_k})$
and $\rho_{\bf r}(w)=\rho_{r_1}(x_{i_1})\cdots\rho_{r_k}(x_{i_k})$ are defined respectively by, for any $r\in\N$,
\begin{eqnarray*}
\tau_r(x_0)=\partial_z^r\frac{x_0}{z}=\frac{-r!x_0}{(-z)^{r+1}}&\hbox{and}&\tau_r(x_1)=\partial_z^r\frac{x_1}{1-z}=\frac{r!x_1}{(1-z)^{r+1}},\\
\rho_r(x_0)=\theta_0^r\frac{(-1)^{-1}x_0}{z}=0&\hbox{and}&\rho_r(x_1)=\theta_0^r\frac{zx_1}{1-z}=\Li^-_{\pi_Y(x_0^{r-1}x_1)}(z)x_1.
\end{eqnarray*}
\end{enumerate}
\end{proposition}

\begin{example}[Coefficients of $\theta_0^n\L$]
\small{Since, for any $u\in X^+$, $\theta_0\Li_{x_0u}=\Li_u$ and $\theta_1\Li_{x_0u}=\Li_0\Li_u$, one obtains for example
\begin{itemize}
\item For any $n\ge1$ and $w\in X^*$, one has $\theta_0^n\Li_{x_0^nw}=\Li_{w}$.
Hence,
$$\theta_0\Li_{x_1}=\Li_{0},\theta_0^2\Li_{x_1}=\Li^-_{\pi_Y( x_1)},\theta_0^3\Li_{x_1}=\Li^-_{\pi_Y (x_0x_1)} \quad \mbox{and} \quad \theta_0^4\Li_{x_1}=\Li^-_{\pi_Y (x_0^2 x_1)}.$$
\item $\theta_0\Li_{x_1^2}=\Li_0\Li_{x_1},\theta_0^2\Li_{x_1^2}=\Li^-_{\pi_Y(x_1)}\Li_{x_1}+\Li_0^2,\theta_0^3\Li_{x_1^2}=\Li^-_{\pi_Y(x_0x_1)}\Li_{x_1}+3\Li^-_{\pi_Y(x_1)}\Li_0$
because
\begin{eqnarray*}
\forall k > 1,&&\theta_0^k\Li_{x_1^2} = \sum^{k-1}_{j=0} \binom{k-1}{j}\Li_{-j}\Li_{2+j-k}. 
\end{eqnarray*}
\end{itemize}}
\end{example}

The noncommutative generating series $\L$ satisfies the differential equation
\begin{eqnarray}\label{drinfeld}
d\L=(x_0\omega_0+x_1\omega_1)\L
\end{eqnarray}
with boundary condition 
\begin{eqnarray}\label{asymptoticbehaviour}
\L(z)\;{}_{\widetilde{z\rightarrow0}}\;\exp(x_0\log z)\quad \mbox{and}\quad
\L(z)\;{}_{\widetilde{z\rightarrow1}}\;\exp(-x_1\log(1-z))\;Z_{\minishuffle}.
\end{eqnarray}
This implies that $\L$ is the exponential of a Lie series \cite{FPSAC98,SLC43}. Hence \cite{orlando},
\begin{eqnarray*}
\log\L=\sum_{k\ge1}\frac{(-1)^{k-1}}k\sum_{u_1,\ldots,u_k\in X^+}\Li_{u_1\shuffle\ldots\shuffle u_k}\;u_1\ldots u_k
=\sum_{w\in X^*}\Li_w\;\pi_1(w).
\end{eqnarray*}

\begin{theorem}[\cite{orlando}]\label{galois}
\begin{enumerate}
\item Let $G, H$ be exponential solutions of (\ref{drinfeld}).
Then there exists a constant Lie series $C$ such that $G = H e^C$.
\item Let $\mathrm{Gal}_{\C}(DE)$ be the differential Galois group associated to the Drinfel'd equation.
Then $\mathrm{Gal}_{\C}(DE)=\{e^C\bv C\in\LXX\}$
and it contains the monodromy group  defined by $\calM_0\L=\L\exp(2\mathrm{i}\pi\mathfrak{m}_0)$ and
$\calM_1\L=\L Z_{\minishuffle}^{-1}\exp(-2\mathrm{i}\pi x_1)Z_{\minishuffle}=\L\exp(2\mathrm{i}\pi\mathfrak{m}_1)$,
where $\mathfrak{m}_0=x_0,\mathfrak{m}_1=\prod_{l\in\Lyn X-X}^{\searrow}\exp(-\zeta(S_l)\ad_{P_l})(-x_1)$.
\end{enumerate}
\end{theorem}

Then let us put\footnote{Here, the coefficient $\scal{B(y_1)}{y_1^k}$ corresponds to the Euler--Mac Laurin constant associated to $\scal{\mathrm{Const}(N)}{y_1^k}$, {\it i.e.} the finite party of its asymptotic expansion in the scale of comparison $\{n^{a}\log^{b}(n)\}_{a\in\Z,b\in\N}$.} $\Lambda:=\pi_Y\L$ and \cite{cade}
\begin{eqnarray}\label{piYL}
&&\mathrm{Mono}(z):=e^{-(x_1+1)\log(1-z)}
=\Sum_{k\ge0}\P_{y_1^k}(z)\;y_1^k\label{Mono}\\
&&\mathrm{Const}:=\sum_{k\ge0}\H_{y_1^k}\;y_1^k
=\exp\biggl(-\Sum_{k\ge1}\H_{y_k}\Frac{(-y_1)^k}{k}\biggr)\label{Const},\\
&&B(y_1):=\exp\biggl(\sum_{k\ge1}\zeta(y_k)\frac{(-y_1)^k}{k}\biggr)\label{monomould},
\end{eqnarray}
and finally, $B'(y_1):=\exp(\gamma y_1)B(y_1)$.
Hence, we get
$\pi_Y\P(z)\;{}_{\widetilde{z\rightarrow1}}\;\mathrm{Mono}(z)\pi_YZ_{\minishuffle}$
and $\H(N)\;{}_{\widetilde{N\rightarrow+\infty}}\;\mathrm{Const}(N)\pi_YZ_{\minishuffle}$
as a consequence of (\ref{piYL})-(\ref{Const}). Or equivalently,

\begin{theorem}[First global renormalizations of divergent polyzetas, \cite{cade}]\label{abel}
\begin{eqnarray*}
\Lim_{z\rightarrow1}\exp\biggl(-y_1\log\frac1{1-z}\biggr)\Lambda(z)=
\Lim_{N\rightarrow+\infty}\exp\biggl(\Sum_{k\ge1}\H_{y_k}(N)\frac{(-y_1)^k}{k}\biggr)\H(N)=\pi_YZ_{\minishuffle}.
\end{eqnarray*}
\end{theorem}

\begin{theorem}[\cite{AofA}]\label{asymtotic}
For any $g\in\calC\{\P_w\}_{w \in Y^*}$, there exist algorithmically computable coefficients
$c_j,b_i\in\C,\alpha_j,\eta_i\in\Z,\beta_j,\kappa_i\in\N$ such that
$$\begin{array}{rclll}
g(z)\;{}_{\widetilde{z\rightarrow1}}\;\Sum_{j=0}^{+\infty}c_j(1-z)^{\alpha_j}\log^{\beta_j}(1-z),&&
\scal{g(z)}{z^n}\;{}_{\widetilde{N\rightarrow+\infty}}\;\Sum_{i=0}^{+\infty}b_in^{\eta_i}\log^{\kappa_i}(n).
\end{array}$$
\end{theorem}

Theorem \ref{asymtotic} means also that the $\{\P_w\}_{w \in Y^*}$ admit a full singular expansion, at $1$,
and then their ordinary Taylor coefficients, $\{\H_w\}_{w \in Y^*}$ admit a full asymptotic expansion,
for $+\infty$. More precisely,

%
\begin{corollary}\label{asymptotic}
For any $w\in X^*$ and for any $k, i, j \in \N$, $k \ge 1$, there exists uniquely determined coefficients $a_{i},b_{i,j}$ belonging to $\calZ$;
$\gamma_{\pi_Y(w)}$, $\alpha_i$ and $\beta_{i,j}$ belonging to the $\Q[\gamma]$-algebra generated by convergent polyzetas such that,
\begin{equation}\label{asy_li}
\Li_w(z)=\Sum_{i=1}^{\abs{w}}a_i\log^{i}(1-z)+\scal{Z_{\minishuffle}}{w}+
\Sum_{j=1}^{k}\Sum_{i=0}^{\abs{w}-1}b_{i,j}\frac{\log^{i}(1-z)}{(1-z)^{-j}}+\mathrm{o}_k^{(1)}((1-z)^{k})
\end{equation}
and, likely
\begin{equation}\label{asy_H}
\H_{\pi_Y(w)}(N)=\sum_{i=1}^{\abs{w}}\alpha_i\log^{i}(N)+\gamma_{\pi_Y(w)}+
\Sum_{j=1}^{k}\Sum_{i=0}^{\abs{w}-1}\beta_{i,j}\frac1{N^j}\log^{i}(N)+\mathrm{o}_k^{(+\infty)}(N^{-k}).
\end{equation}
\end{corollary}
\begin{remark}
i) The two expansions (\ref{asy_li}) and (\ref{asy_H}) are asymptotic expansions of $Li_w$ and $H_w$ with respect to the scales $(1-z)^{n}log(1-z)^m;\ n,m\geq 0$ and $N^{-k}log(N)^m;\ k,m\geq 0$ respectively.\\ 
ii) In (eq. \ref{asy_li}), the error term $\mathrm{o}_k^{(1)}((1-z)^{k})$ can be put to the form $\mathrm{O}_k^{(1)}((1-z)^{k+\epsilon})$ for any $\epsilon\in ]0,1[$. 
\end{remark}

More generally, by Theorem \ref{galois}, we get

\begin{proposition}\label{prim}
For any commutative  $\Q$-algebra $A$ and for any Lie series $C\in\LAX$, we set 
$\overline{\L}=\L e^C,\overline{\Lambda}=\pi_Y\overline{\L}$ and $\overline{\P}(z)=(1-z)^{-1}{\overline{\Lambda}(z)}$, 
then
\begin{enumerate}
\item $\overline{Z}_{\minishuffle}=Z_{\minishuffle}e^C$ is group-like, for the co-product $\Delta_{\minishuffle}$,
\item $\overline{\L}(z)\;{}_{\widetilde{z\rightarrow1}}\;\exp(-x_1\log(1-z))\;\overline{Z}_{\minishuffle}$,
\item $\overline{\P}(z)\;{}_{\widetilde{z\rightarrow1}}\;\mathrm{Mono}(z)\pi_Y\overline{Z}_{\minishuffle}$,
\item $\overline{\H}(N)\;{}_{\widetilde{N\rightarrow\infty}}\;\mathrm{Const}(N)\pi_Y\overline{Z}_{\minishuffle}$,
\end{enumerate}
where, for any $w\in Y^*$ and $N\ge0$, one defines the coefficient $\scal{\overline{\H}(N)}{w}$ of $w$ in
the power series $\overline{\H}(N)$ as the coefficient $\scal{\overline{\P}_w(z)}{z^N}$ of $z^N$
in the ordinary Taylor expansion of the polylogarithmic function $\overline{\P}_w(z)$.
\end{proposition}

By Proposition \ref{prim}, we get successively

\begin{proposition}[\cite{SLC44}]\label{regularized}
Let $\overline{\zeta}_{\minishuffle}$ and $\overline{\zeta}_{\ministuffle}$ be the characters of respectively
$(\AX,\minishuffle)$ and $(\AY,\ministuffle)$ satisfying
$\overline{\zeta}_{\minishuffle}(x_0)=\overline{\zeta}_{\minishuffle}(x_1)=0$
and $\overline{\zeta}_{\ministuffle}(y_1)=0$.
Then
\begin{eqnarray*}
\sum_{w\in Y^*}\overline{\zeta}_{\minishuffle}(w)\;w=&\overline{Z}_{\minishuffle}&=
\Prod_{l\in\Lyn X-X}^{\searrow}\exp(\overline{\zeta}(S_l)\;P_l),\\
\sum_{w\in Y^*}\overline{\zeta}_{\ministuffle}(w)\;w=&\overline{Z}_{\ministuffle}&=
\Prod_{l\in\Lyn Y-\{y_1\}}^{\searrow}\exp(\overline{\zeta}(\Sigma_l)\;\Pi_l).
\end{eqnarray*}
\end{proposition}

\begin{proposition}\label{Euler--Mac Laurin}
Let $\{\overline{\gamma}_w\}_{w \in Y^*}$ be the Euler--Mac Laurin constants associated to
$\{\overline{\H}_w(N)\}_{w \in Y^*}$. Let $\overline{Z}_{\gamma}$ be the noncommutative generating series of these constants.
Then,
\begin{enumerate}
\item The following map realizes a character~:
\begin{eqnarray*}
\overline{\gamma}_{\bullet}:(A\pol{Y},\stuffle)\longrightarrow(\R,.),&&
w\longmapsto\scal{\overline{\gamma}_{\bullet}}{w}=\overline{\gamma}_{w}.
\end{eqnarray*}
\item The noncommutative power series $\overline{Z}_{\gamma}$ is group-like, for $\Delta_{\ministuffle}$.
\item There exists a group-like element $\overline{Z}_{\ministuffle}$, for the co-product $\Delta_{\ministuffle}$,
such that
\begin{eqnarray*}
\overline{Z}_{\gamma}=\sum_{w\in Y^*}\overline{\gamma}_w\;w=\exp(\gamma y_1)\overline{Z}_{\ministuffle}.
\end{eqnarray*}
\end{enumerate}
\end{proposition}

By Theorem \ref{abel}, Propositions \ref{prim} and \ref{Euler--Mac Laurin}, we also get

\begin{proposition}\label{zigzag}
For any $C\in\LAX$ such that $\overline{Z}_{\minishuffle}=Z_{\minishuffle}e^C$. Then
\begin{eqnarray*}
\overline{Z}_{\gamma}=B(y_1)\pi_Y\overline{Z}_{\minishuffle},
&\mbox{or equivalently by cancellation},&
\overline{Z}_{\ministuffle}=B'(y_1)\pi_Y\overline{Z}_{\minishuffle},
\end{eqnarray*}
where $B(y_1)$ and $B'(y_1)$ are given in (\ref{monomould}).
\end{proposition}

By Proposition \ref{prim}, the noncommutative generating series $\overline{Z}_{\minishuffle}$ and $\overline{Z}_{\ministuffle}$
are group-like, for the co-product $\Delta_{\minishuffle}$ and $\Delta_{\ministuffle}$ respectively, and we also have
\begin{eqnarray*}
\overline{Z}_{\minishuffle}&=&\sum_{l\in\Lyn X-X}\overline{\zeta}(S_l)\;P_l
+\sum_{w\notin\Lyn X-X}\overline{\zeta}_{\minishuffle}(S_w)\;P_w,\\
\overline{Z}_{\ministuffle}&=&\sum_{l\in\Lyn Y-\{y_1\}}\overline{\zeta}(\Sigma_l)\;\Pi_l
+\sum_{w\notin\Lyn Y-\{y_1\}}\overline{\zeta}_{\ministuffle}(\Sigma_w)\;\Pi_w.
\end{eqnarray*}
Hence, by Proposition \ref{zigzag}, we deduce in particular,
\begin{eqnarray*}
\sum_{l\in\Lyn Y-\{y_1\}}\overline{\zeta}(\Sigma_l)\;\Pi_l+\ldots=
B'(y_1)\biggl(\sum_{l\in\Lyn X-X}\overline{\zeta}(\pi_YS_l)\;\pi_YP_l+\ldots\biggr).
\end{eqnarray*}
The elements of $\{\pi_YP_l\}_{l\in\Lyn X}$ are decomposable in the linear basis $\{\Pi_w\}_{w\in Y^*}$ of $\calU(\mathrm{Prim}(\calH_{\ministuffle}))$.
Thus, by identification of local coordinates, {\it i.e.} the coefficients of $\{\Pi_l\}_{l\in\Lyn Y-\{y_1\}}$
in the basis $\{\Sigma_l\}_{l\in\Lyn Y-\{y_1\}}$, we get homogenous polynomial
relations on polyzetas encoded by $\{\Sigma_l\}_{l\in\Lyn Y-\{y_1\}}$ \cite{acta}.

\begin{proposition}\label{asymptotic}
There exist $A,B$ and $C\in\Q\left\langle{Y_0}\right\rangle$ such that
\begin{eqnarray*}
\L^{-}(z){}_{\widetilde{z\rightarrow1}}A\odot g\biggl(\Frac1{1-z}\biggr),&
\P^{-}(z){}_{\widetilde{z\rightarrow1}}B\odot\Frac1{1-z}g\biggl(\Frac1{1-z}\biggr),&
\H^{-}(N){}_{\widetilde{N\rightarrow+\infty}}C\odot g(N).
\end{eqnarray*}
where the series $g, h$ were defined in the definition \ref{gandh}.
\end{proposition}

\begin{proof}
By Propositions \ref{noyaux1}, for $w=y_{s_1}\ldots y_{s_r}$, there exists $a,b,c\in\Q$ such that
\begin{eqnarray*}
\Li^-_{w}(z){}_{\widetilde{z\rightarrow1}}\Frac{a}{(1-z)^{\abs{w}+(w)}},&
\P^-_{w}(z){}_{\widetilde{z\rightarrow1}}\Frac{b}{(1-z)^{\abs{w}+(w)+1}},&
\H^-_{w}(N){}_{\widetilde{N\rightarrow+\infty}}cN^{\abs{w}+(w)}.
\end{eqnarray*}
Putting $\scal{A}{w}=(-1)^{\abs{w}}a,\scal{B}{w}=(-1)^{\abs{w}}b,\scal{C}{w}=(-1)^{\abs{w}}c$, it follows the expected results.
\end{proof}


\begin{proposition}\cite{GHN}
For any $w \in Y_0^*$, there are non-zero constants, namely $C^-_w$ and $B^-_w$, which only depend on $w$ and $r$ such that
\begin{eqnarray*}
\lim_{N \to \infty} \Frac{\H^{-}_{w} (N)}{N^{(w)+\abs{w}}C^-_w}=1,&\mbox{i.e.}&\H^{-}_{w} (N)\;{}_{\widetilde{N\rightarrow+\infty}}\;N^{(w)+\abs{w}}C^-_w,\\
\lim\limits_{z \to 1^{-}} \frac{(1-z)^{(w)+\abs{w}}\Li^-_w(z)}{B^-_w}=1,&\mbox{i.e.}&\Li^{-}_{w}(z)\;{}_{\widetilde{z\rightarrow1}}\;\frac{N^{(w)+\abs{w}}B^-_w}{(1-z)^{n+1}}. 
\end{eqnarray*}
Moreover, $C^-_w$ and $B^-_w$ are well determined by
\begin{eqnarray*}
C^-_w=\prod_{w=uv; v\neq 1_{Y_0^*}}\frac1{(v)+\abs{v}}\in\Q&\mbox{and}&B^-_w=((w)+\abs{w})!C^-_w\in\N.
\end{eqnarray*}
\end{proposition}



\begin{example}\cite{GHN}[of $C^-_w$ and $B^-_w$]
\small{$$\begin{tabular}{|c|c|c|c|c|c|}
\hline
$w$& $C^-_w$ & $B^-_w$ & $w$& $C^-_w$ & $B^-_w$\\ 
\hline
$y_0$ &$1$ & $1$ & $y_1y_2$ & ${1}/{15}$ & $8$\\ 
 \hline
$y_1$ & $1/2$ & $1$ & $y_2y_3$ & $1/28$ & $180$\\ 
 \hline
$y_2$ &${1}/{3}$ & $2$ & $y_3y_4$ & ${1}/{49}$ & $8064$\\ 
 \hline
$y_n$ & ${1}/{(n+1)}$ & $n!$ & $y_my_n$ & ${1}/{[(n+1)(m+n+2)]}$ & $n!m!\binom{m+n+1}{n+1}$\\ 
 \hline
$y_0^2$ &${1}/{2}$ & $1$ & $y_2y_2y_3$ & ${1}/{280}$ & $12960$\\ 
 \hline
$y_0^n$ &${1}/{(n!)}$ & $1$ & $y_2y_{10}y_1^2$ & ${1}/{2160}$ & $9686476800$\\ 
 \hline
$y_1^2$ & ${1}/{8}$ & $3$ & $y_2^2y_4y_3y_{11}$ & ${1}/{2612736}$ & $4167611825465088000000$\\ 
 \hline
\end{tabular}$$}
\end{example}

\begin{proposition}\cite{GHN}
Let $u,v\in Y_0^*$. We get $\H^-_u\H^-_v=\H^-_{u\stuffle v}$.
\end{proposition}

\begin{proof}
Let $w \in Y_0^*$ associated to $s = (s_1,\ldots, s_k)$.
The quasi-symmetric monomial functions on the commutative alphabet $t = \{ t_i \}_{i \geq 1}$ are defined as follows 
\begin{eqnarray*}
M_{1_{Y_0^*}}(t) = 1& \mbox{and} & M_w (t)  =  \Sum_{n_1 > \ldots > n_k > 0} t^{s_1}_{n_1} \ldots  t^{s_k}_{n_k},
\end{eqnarray*}
For any $u,v \in Y_0^{*}$, we have $M_u (t)  M_v (t) = M_{u \stuffle v} (t)$.
Then, the harmonic sum $\H^-_{s_1,\ldots,s_k}(N)$  is obtained by specializing the indeterminates
$t = \{  t_i \}_{i \geq 1}$ from $M_w(t)$ as follows:
$t_i = i \mbox{ for } 1 \leq i \leq N$ and $t_i = 0 \mbox{ for } N < i.$
\end{proof}

\begin{theorem}[Second global renormalizations of divergent polyzetas]
\begin{enumerate}
\item The generating series $\H^-$ is group-like and $\log\H^-$ is primitive. Moreover\footnote{Here, the Hadamard product is denoted by $\odot$
and and its dual law, the diagonal comultiplication is denoted by $\Delta_{\odot}$. The series $g, h$ are defined in Definition \ref{gandh}.},
\begin{eqnarray*}
\lim\limits_{N \to+\infty}g^{\odot -1}(N)\odot\H^-(N)=\lim\limits_{z \to 1}h^{\odot -1}((1-z)^{-1})\odot\L^-(z)=C^-.
\end{eqnarray*}
\item $\ker\H^-_{\bullet}$ is a prime ideal of $(\Q \left\langle{Y_0}\right\rangle,\stuffle)$,
{\it i.e.} $\ncp{\Q}{Y_0}\setminus\ker\H^-_{\bullet}$ is closed by $\stuffle$.
\end{enumerate}
\end{theorem}

\begin{proof}
The first result is a consequence of the extended Friedrichs criterion \cite{BDHHT,acta,VJM}
and the second is a consequence of Proposition \ref{asymptotic}.
\end{proof}

\begin{definition}
For any $n\in\N_+$, let $\mathbb{P}_n:=\mathrm{span}_{\R_+}\{w\in Y_0^*|(w)+|w|=n\}\setminus\{0\}$ be
the blunt\footnote{{\it i.e.} without zero or see Appendix A.} convex cone generated by the set $\{w\in Y_0^*|(w)+|w|=n\}$.
\end{definition}

By definition, $C^-_{\bullet}$ is linear on the set $\mathbb{P}_n$. For any $u, v \in Y^*_0$,
one has $u\stuffle v=u\shuffle v+\sum_{|w|<|u|+|v|\atop (w)=(u)+(v)}x_ww$ and the $x_w$'s are positive.
Moreover, for any $w$ which belongs to the support of $\sum_{|w|<|u|+|v|\atop (w)=(u)+(v)}x_ww$, one has $(w)+|w|<(u)+(v)+|u|+|v|$, thus, by the definition of $C^-_{\bullet}$, one obtains
\begin{corollary}\label{tac}
\begin{enumerate}
\item Let $w,v \in Y^*_0$. Then $C^-_wC^-_v=C^-_{w\shuffle v }=C^-_{w\stuffle v }$.
\item For any $P, Q \notin\ker\H^-_{\bullet}$, $C^-_PC^-_Q=  C^-_{P \stuffle Q }$ and
$\Q\left\langle{Y_0}\right\rangle\setminus\ker\H^-_{\bullet}$ is a $\stuffle-$ multiplicative monoid containing $Y^*_0$.
\end{enumerate}
\end{corollary}
 
Now, let us prove that $C^-_{\bullet}$ can be extended as a character, for $\shuffle$, or equivalently,
$C^-$ is group-like (see the Freidrichs' criterion \cite{reutenauer}) and then $\log{C}^-$ is primitive.

\begin{lemma}\label{HGM}
Let $\mathcal{A}$ be an unitary $\R$-associative algebra and $f:\sqcup_{n\ge0}\mathbb{P}_{n}\longrightarrow\mathcal{A}$ such that 
\begin{enumerate}
\item For any $u, v \in Y^*_0$, $f(u \shuffle v)=f(u)f(v)$. In particular, $f(1_{Y^*_0})=1_{\mathcal{A}}$.
\item On every $\mathbb{P}_n$, one has $f(\sum_{i\in I}\alpha_iw_i)=\sum_{i\in I}\alpha_if(w_i)$, where $\alpha_i\in\R^*_+$.
\end{enumerate}
Then $f$ can be extended uniquely as a character,
{\it i.e.} $S_f=\sum_{w\in Y^*_0}f(w) w$ is group-like for $\Delta_{\shuffle}$.
\end{lemma}

\begin{proof}
By definition of $f$ and $S_f$, it is immediate $\scal{S_f}{1_{Y^*_0}}=1_{\mathcal{A}}$. One can check easily that $\Delta_{\shuffle}(S_f)=S_f\otimes S_f$.
Hence, $S_f$ is group-like, for $\Delta_{\shuffle}$.
\end{proof}

\begin{corollary}
The noncommutative generating series $C^-$ is group-like, for $\Delta_{\shuffle}$.
\end{corollary}

\begin{proof}
It is a consequence of Lemma \ref{HGM} and Corollary \ref{tac}.
\end{proof}

\begin{example}\cite{GHN}[of $C^-_{u\shuffle v}$ and $C^-_{u\stuffle v}$]
\small{Let $Y_0 = \{y_i\}_{i \geq 0}$ be an infinite alphabet.
\begin{center}
\begin{tabular}{|c|c|c|c|c|c|}
\hline
$u$ &$C^-_u$&$v$ &$C^-_v$ &$u\shuffle v$& $C^-_{u\shuffle v}$ \\ 
\hline
$y_0$&$1$ & $y_0$&$1$ &$2y^2_0$& $1$ \\ 
 \hline
 $y_0^2$&${1}/{2}$ & & &&  \\ 
 \hline
$y_1$&${1}/{2}$ & $y_2$&${1}/{3}$ &$y_1y_2 +y_2y_1$& ${1}/{6}$ \\ 
 \hline
$y_1y_2$&${1}/{15}$ & $y_2y_1$&${1}/{10}$ && \\ 
 \hline
$y_m$&$(m+1)^{-1}$ & $y_n$&$(n+1)^{-1}$ &$y_my_n +y_ny_m$& $[(m+1)(n+1)]^{-1}$ \\ 
 \hline
$y_my_n$&$\frac{(n+1)^{-1}}{(n+m+2)}$ & $y_ny_m$&$\frac{(m+1)^{-1}}{(m+n+2)}$ && \\ 
 \hline
$y_1$&${1}/{2}$& $y_2y_5$ &${1}/{54}$& $y_1y_2y_5 + y_2y_1y_5 + y_2y_5y_1$& ${1}/{108}$\\
\hline 
$y_1y_2y_5$&${1}/{594}$ & $y_2y_1y_5$&${1}/{528}$ && \\ 
 \hline
$y_2y_5y_1$&${1}/{176}$ & & && \\ 
 \hline
$y_0y_1$&${1}/{6}$ & $y_2y_3$&${1}/{28}$ &$y_0y_1y_2y_3 + y_0y_2y_1y_3$ & ${1}/{168}$ \\ 
& & & &$+ y_0y_2y_3y_1 + y_2y_3y_0y_1$&  \\ 
& & & &$+ y_2y_0y_1y_3 + y_2y_0y_3y_1$&  \\ 
 \hline
$y_0y_1y_2y_3$&${1}/{2520}$ & $y_0y_2y_1y_3$&${1}/{2160}$ && \\ 
 \hline
$y_0y_2y_3y_1$&${1}/{1080}$ & $y_2y_3y_0y_1$&${1}/{420}$ && \\ 
 \hline
$y_2y_0y_1y_3$&${1}/{1680}$ & $y_2y_0y_3y_1$&${1}/{840}$ && \\ 
 \hline
$y_ay_b$&$\frac{(b+1)^{-1}}{(a+b+2)}$ & $y_cy_d$&$\frac{(d+1)^{-1}}{(c+d+2)}$ &$y_ay_by_cy_d + y_ay_cy_by_d$ & $\frac{(b+1)^{-1}(d+1)^{-1}}{(a+b+2)(c+d+2)}$ \\ 
& & & &$+ y_ay_cy_dy_b + y_cy_dy_ay_b$&  \\ 
& & & &$+ y_cy_ay_by_d + y_cy_ay_dy_b$&  \\ 
 \hline
\end{tabular}
\end{center}
\begin{center}
\begin{tabular}{|c|c|c|c|c|c|}
\hline
$u$ &$C^-_u$&$v$ &$C^-_v$ &$u\stuffle v$& $C^-_{u\stuffle v}$ \\ 
\hline
$y_0$&$1$ & $y_0$&$1$ &$2y^2_0 + y_0$& $1$ \\ 
 \hline
$y_1$&${1}/{2}$ & $y_2$&${1}/{3}$ &$y_1y_2 +y_2y_1 + y_3$& ${1}/{6}$ \\ 
 \hline
$y_m$&$(m+1)^{-1}$ & $y_n$&$(n+1)^{-1}$ &$y_my_n +y_ny_m + y_{n+m}$& $[(m+1)(n+1)]^{-1}$ \\ 
 \hline
$y_1$&${1}/{2}$& $y_2y_5$ &${1}/{54}$& $y_1y_2y_5 + y_2y_1y_5 + y_2y_5y_1$& ${1}/{108}$\\ 
& & & &$+y_3y_5 + y_2y_6$& \\ 
 \hline
$y_0y_1$&${1}/{6}$ & $y_2y_3$&${1}/{28}$ &$y_0y_1y_2y_3 + y_0y_2y_1y_3$ & ${1}/{168}$ \\ 
& & & &$+ y_0y_2y_3y_1 + y_2y_3y_0y_1$&  \\ 
& & & &$+ y_2y_0y_1y_3 + y_2y_0y_3y_1 + y_0y_2y_4$&  \\ 
& & & & $+ y_0y_3^2+ y_2y_3y_1 + y_2y_1y_3$&  \\ 
& & & & $+ y_2y_0y_4 + y_2y_3y_1 + y_2y_4$&  \\
 \hline
$y_ay_b$&$\frac{(b+1)^{-1}}{(a+b+2)}$ & $y_cy_d$&$\frac{(d+1)^{-1}}{(c+d+2)}$ &$y_ay_by_cy_d + y_ay_cy_by_d$ & $\frac{(b+1)^{-1}(d+1)^{-1}}{(a+b+2)(c+d+2)}$ \\ 
& & & &$+ y_ay_cy_dy_b + y_cy_dy_ay_b + y_cy_ay_by_d$&  \\ 
& & & &$+ y_cy_ay_dy_b + y_ay_cy_{b+d} + y_ay_{b+c}y_d$&  \\ 
& & & &$+ y_cy_ay_{b+d} + y_cy_{a+d}y_b$&  \\ 
& & & &$+ y_{a+c}y_by_d + y_{a+c}y_dy_b + y_{a+c}y_{b+d}$&  \\
 \hline
\end{tabular}
\end{center}
In the above tables, it is clearly seen that $C^-_{\bullet}$ is linear on $\mathbb{P}_n$.
For example, let $u = y_1$ and $v = y_2y_5$. Then $u \shuffle v=y_1y_2y_5+y_2y_1y_5+y_2y_5y_1$.
Hence, we get 
$C^-_{y_1y_2y_5}+C^-_{ y_2y_1y_5}+C^-_{ y_2y_5y_1}=\frac{1}{594}+\frac{1}{528}+\frac{1}{176}=\frac{1}{108}
=C^-_{y_1}C^-_{ y_2y_5}=C^-_{y_1 \shuffle y_2y_5}$.
Note that $y_1y_2y_5,y_2y_1y_5,y_2y_5y_1\in\mathbb{P}_{11}$.
But we have also $u \stuffle v = y_1y_2y_5 + y_2y_1y_5 + y_2y_5y_1 +y_3y_5 + y_2y_6$. Moreover,
$C^-_{y_1y_2y_5}+ C^-_{ y_2y_1y_5}+ C^-_{ y_2y_5y_1}+C^-_{y_3y_5}+C^-_{y_2y_6}=\frac{1}{108}+\frac{13}{420}\neq \frac{1}{108}=C^-_{y_1} C^-_{y_2y_5}.$
However, from $y_3y_5,y_2y_6 \in\mathbb{P}_{10}$, we can conclude that
\begin{eqnarray*}
C^-_{y_1 \stuffle y_2y_5} = C^-_{y_1y_2y_5 + y_2y_1y_5 + y_2y_5y_1 +y_3y_5 + y_2y_6} = C^-_{y_1y_2y_5 + y_2y_1y_5 + y_2y_5y_1} = {1}/{108} = C^-_{y_1}C^-_{y_2y_5}.
\end{eqnarray*}}
\end{example}

\section{Polysystems and differential realization}\label{Polysystem}

\subsection{Polysystems and convergence criterion}\label{convergence}

\subsubsection{Estimates (from above) for series}
Here, $(\K,\absv{.})$ is a normed space.

\begin{definition}{\rm (\cite{these,cade})}\label{xi-em}\label{chi-gc}
Let $\xi,\chi$ be real positive functions over $ X^*$.
Let $S\in\KXX$.
\begin{enumerate}
\item $S$ will be said {\it $\xi-$ exponentially bounded from above} if it satisfies
\begin{eqnarray*}
\exists K\in\R_+,\exists n\in\N,\forall w\in X^{\ge n},&&
\absv{\langle S\bv w\rangle}\le K{\xi(w)}/{\abs w!}.
\end{eqnarray*}

We denote by $\K^{\xi-\mathrm{em}}\ser{X}$ the set of formal power series
in $\KXX$ which are $\xi-$ exponentially bounded from above.
\item $S$ satisfies the {\it $\chi-$growth condition} if it satisfies
\begin{eqnarray*}
\exists K\in\R_+,\exists n\in\N,\forall w\in X^{\ge n},&&\absv{\langle S\bv w\rangle}\le K\chi(w)\abs w!.
\end{eqnarray*}
We denote by $\K^{\chi-\mathrm{gc}}\ser{X}$ the set of formal power series in $\KXX$ satisfying the $\chi-$growth condition.
\end{enumerate}
\end{definition}

\begin{lemma}\label{R}
If $R=\Sum_{w\in X^*}\abs{w}!\;w$ then $\scal{R^{\minishuffle 2}}{w}=\Sum_{u,v\in X^*\atop\supp(u\shuffle v)\ni w}\abs{u}!\abs{v}!\le2^{\abs{w}}\abs{w}!$.
\end{lemma}

\begin{proof}
One has
\begin{eqnarray*}
\sum_{u,v\in X^*\atop\supp(u\shuffle v)\ni w}\abs{u}!\abs{v}!
=\sum_{k=0}^{\abs w}\sum_{\abs u=k,\abs v=\abs w-k\atop\supp(u\shuffle v)\ni w}k!(\abs{w}-k)!
=\sum_{k=0}^{\abs w}{\abs w\choose k}k!(\abs{w}-k)!=\sum_{k=0}^{\abs w}\abs w!.
\end{eqnarray*}
The last sum is equal to $(1+\abs w)\abs w!$.
By induction on $|w|$, one has $1+|w|\le2^{|w|}$. Then the expected result follows.
\end{proof}

\begin{proposition}\label{shufflegrowth}
If $S_1, S_2$ satisfy the growth condition then $S_1+S_2,S_1\shuffle S_2$ do also.
\end{proposition}

\begin{proof}
It is immediate for $S_1+S_2$. Next, since $\absv{\scal{S_i}{w}}\le K_i\chi_i(w)\abs{w}!$ then
\begin{eqnarray*}
\scal{S_1\shuffle S_2}{w}&=&\sum_{\supp(u\shuffle v)\ni w}\scal{S_1}{u}\scal{S_2}{v},\\
\Rightarrow\quad\absv{\langle S_1\shuffle S_2\bv w\rangle}
&\le&K_1K_2\sum_{u,v\in X^*\atop\supp(u\shuffle v)\ni w}(\chi_1(u)\abs{u}!)(\chi_2(v)\abs{v}!).
\end{eqnarray*}
Let $K=K_1K_2$ and let $\chi$ be a real positive function over $ X^*$ such that, for any $w\in X^*$
\begin{eqnarray*}
\chi(w)=\max\{\chi_1(u)\chi_2(v)\bv u,v\in X^*&\mbox{and}&supp(u\shuffle v)\ni w\}.
\end{eqnarray*}
With the notations in Lemma \ref{R}, we get $\absv{\scal{S_1\shuffle S_2}{w}}\le K\chi(w)\scal{S_1R^{\minishuffle 2}}{w}$.
Hence, $S_1\shuffle S_2$ satisfies the $\chi'$-growth condition with $\chi'(w)=2^{\abs{w}}\chi(w).$
\end{proof}

\begin{definition}{\rm (\cite{these,cade})}
Let $\xi$ be a real positive function defined over $X^*$, $S$ will be said {\it $\xi$-exponentially continuous} if it is
continuous over $\K^{\xi-\mathrm{em}}\ser{X}$. The set of formal power series which are $\xi$-exponentially continuous is denoted by 
$\K^{\xi-ec}\ser{X}.$
\end{definition}

\begin{lemma}{\rm \cite{these,cade}}\label{exponentiallycontinuous}
For any real positive function $\xi$ defined over $ X^*$, we have $\KX\subset\K^{\xi-ec}\ser{X}$.
Otherwise, for $\xi=0$, we get $\KX=\K^{0-ec}\ser{X}$. Hence, any polynomial is $0-$exponentially continuous.
\end{lemma}

\begin{proposition}[\cite{these,cade}]\label{convergencecriterion}
Let $\xi,\chi$ be real positive functions over $ X^*$ and $P\in\KX$.
\begin{enumerate}
\item Let $S\in\K^{\xi-\mathrm{em}}\ser{X}$. The right residual of $S$ by $P$ belongs to $\K^{\xi-\mathrm{em}}\ser{X}$.
\item Let $R\in\K^{\chi-\mathrm{gc}}\ser{X}$. The concatenation $SR$ belongs to $\K^{\chi-\mathrm{gc}}\ser{X}$.
\item Moreover, if $\xi$ and $\chi$ are morphisms over $X^*$ satisfying
$\sum_{x\in X}\chi(x)\xi(x)<1$
then, for any $F\in\K^{\chi-\mathrm{gc}}\ser{X}$, $F$ is continuous over $\K^{\xi-\mathrm{em}}\ser{X}$.
\end{enumerate}
\end{proposition}

\begin{proof}
\begin{enumerate}
\item Since $S\in\K^{\xi-\mathrm{em}}\ser{X}$ then 
\begin{eqnarray*}
\exists K\in\R_+,\exists n\in\N,\forall w\in X^{\ge n},&&\absv{\scal{S}{w}}\le K{\xi(w)}/{\abs w!}.
\end{eqnarray*}
If $u\in\supp(P)$ then, for any $w\in X^*$,
one has $\scal{S\triangleright u}{w}=\scal{S}{uw}$
and $S\triangleright u$ belongs to $\K^{\xi-\mathrm{em}}\ser{X}$:
\begin{eqnarray*}
\exists K\in\R_+,\exists n\in\N,\forall w\in X^{\ge n},&&\absv{\scal{S\triangleright u}{w}}\le[K\xi(u)]{\xi(w)}/{\abs w!}.
\end{eqnarray*}
It follows that $S\triangleright P$ is $\K^{\xi-\mathrm{em}}\ser{X}$ by taking $K_1=K\max_{u\in\supp(P)}\xi(u)$.

\item Since $R\in\K^{\chi-\mathrm{gc}}\ser{X}$ then
\begin{eqnarray*}
\exists K\in\R_+,\exists n\in\N,\forall w\in X^{\ge n},&&\absv{\scal{S}{w}}\le K\chi(w)\abs w!.
\end{eqnarray*}
Let $v\in\supp(P)$ such that $v\neq\epsilon$.
Since $Rv$ belongs to $\K^{\chi-\mathrm{gc}}\ser{X}$ and one has, for $w\in X^*$,
$\scal{Rv}{w}=\scal{R}{v\resg w}$, {\it i.e.} there exists $K\in\R_+,n\in\N$ such that
\begin{eqnarray*}
\absv{\scal{R}{v\resg w}}\le{K}\chi(v\resg w)(\abs w-\abs v)!\le{K}\abs w{\chi(w)}/{\chi(v)}.
\end{eqnarray*}
Note if $v\resg w=0$ then $\langle Rv\bv w\rangle=0$ and the previous conclusion holds.
It follows that $RP$ is $\K^{\chi-\mathrm{gc}}\ser{X}$ by taking $K_2=K\min_{v\in\supp(P)}\chi(v)^{-1}$.

\item Let $\xi,\chi$ be functions which satisfy the upper bound condition. The following quantity is well defined
\begin{eqnarray*}
\sum_{w\in X^*}\chi(w)\xi(w)=\biggl(\sum_{x\in X}\chi(x)\xi(x)\biggr)^*.
\end{eqnarray*}
If $F\in\K^{\chi-\mathrm{gc}}\ser{X},C\in\K^{\xi-\mathrm{em}}\ser{X}$ then there
exist $K_i\in\R_+,n_i\in\N,i=1,2$ such that, for $w\in X^{\ge n_i}$,
$\absv{\langle F\bv w\rangle}\le K_1\chi(w)\abs w!$ and $\absv{\langle C\bv w\rangle}\le K_2{\xi(w)}/{\abs w!}$.
Thus,
\begin{eqnarray*}
&&\forall w\in X^*, \abs w\ge\max\{n_1,n_2\},\quad
\absv{\langle F|w\rangle\langle C|w\rangle}\le K_1K_2\chi(w)\xi(w),\\
\Rightarrow&&\sum_{w\in X^*}\absv{\langle F|w\rangle\langle C|w\rangle}
\le K_1K_2\sum_{w\in X^*}\chi(w)\xi(w)=K_1K_2\biggl(\sum_{x\in X}\chi(x)\xi(x)\biggr)^*.
\end{eqnarray*}
\end{enumerate}
\end{proof}

\subsubsection{Upper bounds {\it \`a la} Cauchy}

The algebra of formal power series on commutative indeterminates $\{q_1,\ldots,q_n\}$
with coefficients in $\C$ is denoted by $\C[\![q_1,\ldots,q_n]\!]$.

\begin{definition}{\rm (\cite{these,cade})}
Let $f=\in\C[\![q_1,\ldots,q_n]\!]$.
We set
\begin{eqnarray*}
E(f)
 &:=&\{\rho\in\R_+^n:\exists C_f\in\R_+\mbox{ s.t. }
 \forall i_1,\ldots,i_n\ge0,\abs{f_{i_1,\ldots,i_n}}\rho_1^{i_1}\ldots\rho_n^{i_n}\le C_f\}.\cr
 {\breve E}\!(f)&:\hskip 0,25 cm&\mbox{the interior of $E(f)$ in }\R^n.\cr
\Conv(f)&:=&\{q\in\C^n:(\abs{q_1},\ldots,\abs{q_n})\in{\breve E}\!(f)\}:\quad
\mbox{the convergence domain of }f.
 \end{eqnarray*}
$f$ is {\it convergent} if $\Conv(f)\ne\emptyset$. Let $\calU\subset\C^n$ be an open domain and
$q\in\C^n$. $f$ is convergent on $q$ (resp. over $\calU$) if $q\in\Conv(f)$ (resp. $\calU\subset \Conv(f)$).
We set $\C^{\rm cv}[\![q_1,\ldots,q_n]\!]:=\{f\in\C[\![q_1,\ldots,q_n]\!]:\Conv(f)\ne\emptyset\}$.
Let $q\in\Conv(f)$. There exist $C_f \in \R_+ ,\rho \in E(f),{\bar\rho} \in {\breve E}\!(f)$ such that
$\abs{q_1}<{\bar\rho}_1<\rho_1,\ldots,\abs{q_n}<{\bar\rho}_n<\rho_n$ and $\abs{f_{i_1,\ldots,i_n}}\rho_1^{i_1}\ldots\rho_n^{i_n}\le C_f$,
for $i_1,\ldots,i_n\ge0.$ 

The {\it convergence modulus} of $f$ at $q$ is $(C_f,\rho,{\bar\rho})$.
\end{definition}

Suppose $\Conv(f)\ne\emptyset$ and let $q\in\Conv(f)$. If $(C_f,\rho,{\bar\rho})$ is a convergence modulus of $f$ at $q$
then $\abs{f_{i_1,\ldots,i_n}q_1^{i_1}\ldots q_n^{i_n}}\le C_f({\bar\rho_1}/{\rho_1})^{i_1}\ldots({\bar\rho_1}/{\rho_1})^{i_n}$.
Hence, at $q$, $f$ is majored termwise by $C_f\prod_{k=0}^m(1-{\bar\rho_k}/{\rho_k})^{-1}$
and it is uniformly absolutely convergent in $\{q\in\C^n:\abs{q_1}<\;{\bar\rho},\ldots,\abs{q_n}<{\bar\rho}\}$ which is open in $\C^n$.
Thus, $\Conv(f)$ is open in $\C^n$. Since the partial derivation $D^{j_1}_1\ldots D^{j_n}_nf$ is estimated by 
\begin{eqnarray*}
\absv{D^{j_1}_1\ldots D^{j_n}_nf}\le
C_f\frac{\partial^{j_1+\ldots+j_n}}{\partial^{j_1} \bar{\rho}_1 \ldots \partial^{j_n} \bar{\rho}_n}
\prod_{k=0}^m\biggl(1-\frac{\bar\rho_k}{\rho_k}\biggr)^{-1}.
\end{eqnarray*}

\begin{proposition}{\rm (\cite{these})}
We have $\Conv(f)\subset\Conv(D^{j_1}_1\ldots D^{j_n}_nf)$.
\end{proposition}

Let $f\in\C^{\rm cv}[\![q_1,\ldots,q_n]\!]$.
Let $\{A_i\}_{i=0,1}$ be a polysystem defined as follows
\begin{eqnarray}\label{vectorfield}
A_i =\sum_{j=1}^nA_i^j\frac{\partial}{\partial q_j},&\forall j=1,\ldots,n,&A_i^j(q)\in\C^{\rm cv}[\![q_1,\ldots,q_n]\!].
\end{eqnarray}
Let $(\rho,\bar\rho,C_f),\{(\rho,\bar\rho,C_i)\}_{i=0,1}$ be convergence modulus at
$q\in\Conv(f)\cap_{i=0,1, j=1,\ldots,n}\Conv(A_i^j)$ of $f$ and $\{A_i^j\}_{j=1,\ldots,n}$.
Let us consider the following monoid morphisms
\begin{eqnarray}
\calA(1_{X^*})=\mbox{identity}&\mbox{and}&C(1_{X^*})=1,\label{calA1}\\
\forall w=vx_i,x_i\in X,v\in X^*,\quad 
\calA(w)=\calA(v)A_i&\mbox{and}&C(w)=C(v)C_i\label{calA2}.
\end{eqnarray}

\begin{lemma}{\rm (\cite{fliessrealisation})}\label{coefficients}
For $i=0,1$ and $j=1,\ldots,n$, one has $A_i\circ q_j=A_i^j$. Hence,
\begin{eqnarray*}
\forall i=0,1,&&A_i =\sum_{j=1}^n(A_i\circ q_j)\frac{\partial}{\partial q_j}.
\end{eqnarray*}
\end{lemma}

\begin{lemma}{\rm (\cite{fliess1})}
For any word $w$, $\calA(w)$ is continuous over $\C^{\rm cv}[\![q_1,\ldots,q_n]\!]$
and, for any $f,g\in\C^{\rm cv}[\![q_1,\ldots,q_n]\!]$, one has
\begin{eqnarray*}
\calA(w)\circ(fg)=\sum_{u,v\in X^*}\scal{u\shuffle v}{w}(\calA(u)\circ f)(\calA(v)\circ g).
\end{eqnarray*}
\end{lemma}
These notations are extended, by linearity, to $\KX$ and  we will denote
$\calA(w)\circ f_{|q}$ the evaluation of $\calA(w)\circ f$ at $q$.

\begin{definition}{\rm (\cite{fliess1})}\label{fliess}
Let $f\in\C^{\rm cv}[\![q_1,\ldots,q_n]\!]$.
The generating series of the polysystem $\{A_i\}_{i=0,1}$ and of the observation $f$ is given by
\begin{eqnarray*}
\sigma f:=\sum_{w\in X^*}\calA(w)\circ f\;w\quad\in\quad\serie{\C^{\rm cv}[\![q_1,\ldots,q_n]\!]}{X}.
\end{eqnarray*}
Then the following generating series is called {\it Fliess generating series}
of the polysystem $\{A_i\}_{i=0,1}$ and of the observation $f$ at $q$:
\begin{eqnarray*}
\sigma f_{|_{q}}:=\sum_{w\in X^*}\calA(w)\circ f_{|q}\;w\quad\in\quad\serie{\C}{X}.
\end{eqnarray*}
\end{definition}

\begin{lemma}{\rm (\cite{fliess1})}\label{sigmamorphism}
The map $\sigma:(\C^{\rm cv}[\![q_1,\ldots,q_n]\!],.)\longrightarrow(\serie{\C^{\rm cv}[\![q_1,\ldots,q_n]\!]}{X},\shuffle)$
is an algebra morphism, {\it i.e.} for any $f,g\in\C^{\rm cv}[\![q_1,\ldots,q_n]\!]$ and $\mu,\nu\in\C$, one has
$\sigma(\nu f+\mu h)=\nu\sigma f+\mu\sigma g$ and $\sigma(fg)=\sigma f\shuffle\sigma g$.
\end{lemma}

\begin{lemma}{\rm (\cite{fliessrealisation})}\label{fields}
For any $w\in X^*$, $\sigma (\calA(w)\circ f)=w\triangleleft\sigma f\in\serie{\C^{\rm cv}[\![q_1,\ldots,q_n]\!]}{X}$.
\end{lemma}

\begin{theorem}{\rm (\cite{these})}\label{growthcondition}
\begin{enumerate}
\item Let $\tau=\min_{1\le k\le n}\rho_k$ and $r=\max_{1\le k\le n}{\bar\rho_k}/{\rho_k}$. We have
\begin{eqnarray*}
\absv{\calA(w)\circ f}&\le&C_f\frac{(n+1)}{(1-r)^n}\Frac{C(w)\abs w!}{{{n+\abs{w}-1}\choose{\abs w}}}
\biggl[\frac{n}{\tau(1-r)^{n+1}}\biggr]^{\abs{w}}\\
&\le&C_f\frac{(n+1)}{(1-r)^n}C(w)
\biggl[\frac{n}{\tau(1-r)^{n+1}}\biggr]^{\abs{w}}\abs w!.
\end{eqnarray*}
\item Let $K=C_f{(n+1)}{(1-r)^{-n}}$ and $\chi$ be the real positive function defined over $X^*$~:
\begin{eqnarray*}
\forall i=0,1,&&\chi(x_i)={C_in}{(1-r)^{-(n+1)}}/\tau.
\end{eqnarray*}
Then\footnote{It is the same for the Fliess generating series $\sigma f_{|q}$ of $\{A_i\}_{i=0,1}$ and of $f$ at $q$.}
the generating series $\sigma f$ of the polysystem $\{A_i\}_{i=0,1}$ and of the observation $f$ satisfies the $\chi-$growth condition.
\end{enumerate}
\end{theorem}

\subsection{Polysystem and nonlinear differential equation}\label{polysystem}

\subsubsection{Nonlinear differential equation (with three singularities)}\label{system}

Let us consider the singular inputs\footnote{These singular inputs are not
included in the studies of Fliess motivated, in particular, by the renormalization of $y$ at $+\infty$ \cite{fliess1,fliessrealisation}.}
$u_0(z):=z^{-1}$ and $u_1(z):=(1-z)^{-1}$, and
\begin{eqnarray}\label{nonlinear}
\left\{\begin{array}{lcl}
y(z)&=&f(q(z)),\\
\dot q(z)&=&A_0(q)\;u_0(z)+A_1(q)\;u_1(z),\\
q(z_0)&=&q_0,
\end{array}\right.
\end{eqnarray}
where the state $q=(q_1,\ldots,q_n)$ belongs to a complex analytic manifold of dimension $n$,
$q_0$ is the initial state, the observation $f$ belongs to $\C^{\rm cv}[\![q_1,\ldots,q_n]\!]$
and $\{A_i\}_{i=0,1}$ is the polysystem defined on (\ref{vectorfield}).

\begin{definition}{\rm (\cite{hoangjacoboussous})}
The following power series is called{ \it transport operator}\footnote{It plays the r\^ole of the resolvent in Mathematics and the evolution operator in Physics.} of the polysystem $\{A_i\}_{i=0,1}$ and of the observation $f$
\begin{eqnarray*}
\calT:=\sum_{w\in X^*}\alpha_{z_0}^z(w)\;\calA(w).
\end{eqnarray*}

\end{definition}
By the factorization of the monoid by Lyndon words, we have \cite{hoangjacoboussous}
\begin{eqnarray*}
\calT=(\alpha_{z_0}^z\otimes\calA)\biggl(\sum_{w\in X^*}w\otimes w\biggr)=\prod_{l\in\Lyn X}\exp[\alpha_{z_0}^z(S_l)\;\calA(P_l)].
\end{eqnarray*}

The Chen generating series along the path $z_0\path z$, associated to $\omega_0,\omega_1$ is
\begin{eqnarray}\label{chen}
S_{z_0\path z}:=\sum_{w\in X^*}\langle S\bv w\rangle\;w&\mbox{with}&\langle S\bv w\rangle=\alpha_{z_0}^z(w)
\end{eqnarray}
which solves the differential equation (\ref{drinfeld}) with the initial condition $S_{z_0\path z_0}=1$.
Thus, $S_{z_0\path z}$ and $\L(z)\L(z_0)^{-1}$ satisfy the same differential equation
taking the same value at $z_0$ and $S_{z_0\path z}=\L(z)\L(z_0)^{-1}$.
Any Chen generating series $S_{z_0\path z}$ is group like \cite{ree}
and depends only on the homotopy class of $z_0\path z$ \cite{chen}.
The product of $S_{z_1\path z_2}$ and $S_{z_0\path z_1}$ is 
$S_{z_0\path z_2}=S_{z_1\path z_2}S_{z_0\path z_1}$.
Let $\eps\in]0,1[$ and $z_i=\eps\exp({\mathrm i}\beta_i)$, for $i=0,1$. We set $\beta=\beta_1-\beta_0$.
Let $\Gamma_0(\eps,\beta_0)$ (resp. $\Gamma_1(\eps,\beta_1)$) be the path turning around $0$ (resp. $1$)
in the positive direction from $z_0$ to $z_1$. By induction on the length of $w$, one has
$\abs{\langle S_{\Gamma_i(\eps,\beta)}\bv w\rangle}=(2\eps)^{\abs{w}_{x_i}}{\beta^{\abs{w}}}/{\abs{w}!}$,
where $\abs{w}$ denotes the length of $w$ and $\abs{w}_{x_i}$ denotes the number of occurrences of letter $x_i$ in $w$, for $i=0$ or $1$.
When $\eps$ tends to $0^+$, these estimations yield $S_{\Gamma_i(\eps,\beta)}=e^{{\mathrm i}\beta x_i}+o(\eps)$.
In particular, if $\Gamma_0(\eps)$ (resp. $\Gamma_1(\eps)$) is a
circular path of radius $\eps$ turning around $0$ (resp. $1$) in
the positive direction, starting at $z=\eps$ (resp. $1-\eps$), then, by the
noncommutative residue theorem \cite{FPSAC98,SLC43}, we get
\begin{eqnarray}\label{majoration}
S_{\Gamma_0(\eps)}=e^{2{\mathrm i}\pi x_0}+o(\eps)&\mbox{and}&S_{\Gamma_1(\eps)}=e^{-2{\mathrm i}\pi x_1}+o(\eps).
\end{eqnarray}
Finally, the asymptotic behaviors of $\L$ on (\ref{asymptoticbehaviour}) give \cite{SLC43,FPSAC98}
\begin{eqnarray}\label{chenregularization}
S_{\eps\path1-\eps}&_{\widetilde{\eps\rightarrow0^+}}&e^{-x_1\log\eps}Z_{\minishuffle}\;e^{-x_0\log\eps}.
\end{eqnarray}

In other terms, $Z_{\minishuffle}$ is the regularized Chen generating series
$S_{\eps\path1-\eps}$ of diffferential forms $\omega_0$ and $\omega_1$:
$Z_{\minishuffle}$ is the noncommutative generating series of the finite parts of the
coefficients of the Chen generating series $e^{x_1\log\eps}\;S_{\eps\path1-\eps}\;e^{x_0\log\eps}$.

\subsubsection{Asymptotic behavior via extended Fliess fundamental formula}

\begin{theorem}{\rm (\cite{cade})}\label{fondamentalformula}
$y(z)=\calT\circ f_{|_{q_0}}=\scal{\sigma f_{|_{q_0}}}{S_{z_0\path z}}$.
\end{theorem}
This extends then Fliess’ fundamental formula \cite{fliess1}. By Theorem \ref{factorization},
the expansions of the output $y$ of nonlinear dynamical system with singular inputs follow
\begin{corollary}[Combinatorics of Dyson series]\label{output}
\begin{eqnarray*}
y(z)&=&\sum_{w\in X^*}g_w(z)\;\calA(w)\circ f_{|q_0}\cr
&=&\sum_{k\ge0}\sum_{n_1,\ldots,n_k\ge0}g_{x_0^{n_1}x_1\ldots x_0^{n_k}x_1}(z)\;\ad_{A_0}^{n_1}A_1\ldots\ad_{A_0}^{n_k}A_1e^{\log zA_0}\circ f_{|q_0}\cr
&=&\prod_{l\in\Lyn X}\exp\biggl(g_{S_l}(z)\;\calA(P_l)\circ f_{|q_0}\biggr)\cr
&=&\exp\biggl(\sum_{w\in X^*}g_w(z)\;\calA(\pi_1(w))\circ f_{|_{q_0}}\biggr),
\end{eqnarray*}
where, for any word $w$ in $X^*,g_w$ belongs to the polylogarithm algebra.
\end{corollary}

Since $S_{z_0\path z}=\L(z)\L(z_0)^{-1}$ and $\sigma f_{|_{q_0}},\L(z_0)^{-1}$
are invariant by $\partial_z=d/dz$ and $\theta_0=zd/dz$ then
we get the $n$-th order differentiation of $y$,  with respect to $\partial_z$ and $\theta_0$~:
\begin{eqnarray*}
\partial_z^ny(z)=\scal{\sigma f_{|_{q_0}}}{\partial^nS_{z_0\path z}}=\scal{\sigma f_{|_{q_0}}}{\partial_z^n\L(z)\L(z_0)^{-1}},\\
\theta_0^ny(z)=\scal{\sigma f_{|_{q_0}}}{\theta_0^nS_{z_0\path z}}=\scal{\sigma f_{|_{q_0}}}{\theta_0^n\L(z)\L(z_0)^{-1}}.
\end{eqnarray*}
With the notations of Proposition \ref{lem:derivL}, we get respectively
\begin{eqnarray*}
\partial_z^ny(z)=\scal{\sigma f_{|_{q_0}}}{[D_n(z)\L(z)]\L(z_0)^{-1}}=\scal{\sigma f_{|_{q_0}}\resd D_n(z)}{\L(z)\L(z_0)^{-1}},\\
\theta_0^ny(z)=\scal{\sigma f_{|_{q_0}}}{E_n(z)\L(z)]\L(z_0)^{-1}}=\scal{\sigma f_{|_{q_0}}\resd E_n(z)}{\L(z)\L(z_0)^{-1}}.
\end{eqnarray*}

For $z_0=\eps\rightarrow0^+$, the asymptotic behavior and the renormalization at $z=1$ of $\partial_z^ny$
and $\theta_0^ny$ (or the asymptotic expansion and the renormalization of its Taylor coefficients at $+\infty$)
are deduced from (\ref{chenregularization}) and extend a little  bit results of \cite{cade,acta}~:

\begin{corollary}[Asymptotic behavior of output]\label{asymptoticofoutput}
\begin{enumerate}
\item The $n$-order differentiation of the output $y$ of the system (\ref{nonlinear})
is a $\calC$-combination of the elements $g$ belonging to the polylogarithm algebra
and\footnote{Moreover, we get more out of this i.e.
$\theta_1^ny(z)=\langle\sigma f_{|_{q_0}}\bbv\theta_1^nS_{z_0\path z}\rangle
=\langle\sigma f_{|_{q_0}}\bbv\theta_1^n\L(z)\L(z_0)^{-1}\rangle$.
Therefore, $\theta_1^ny(z)=\scal{\sigma f_{|_{q_0}}}{[D_n(z)-E_n(z)]\L(z)\L(z_0)^{-1}}
=\scal{\sigma f_{|_{q_0}}\resd[D_n(z)-E_n(z)]}{\L(z)\L(z_0)^{-1}}$. Hence,
$\theta_1^ny(1)\;{}_{\widetilde{\eps\rightarrow0^+}}\;\sum_{w\in X^*}\scal{\calA(w)\circ f_{|_{q_0}}}{w}
\scal{[D_n(1-\eps)-E_n(1-\eps)]e^{-x_1\log\eps}\;Z_{\minishuffle}\;e^{-x_0\log\eps}}{w}$.

The actions of $\theta_0=u_0(z)^{-1}d/dz$ and $\theta_1=u_1(z)^{-1}d/dz$ over $y$ are equivalent to those of the residuals of $\sigma f_{|_{q_0}}$ 
by respectively $x_0$ and $x_1$. They correspond to {\it functional} differentiations \cite{causaldifferentiation}
while $\partial_z=d/dz$ is the ordinary differentiation and is equivalent to the residual by $x_0+x_1$.}, for any $n\ge0$,
\begin{eqnarray*}
\partial_z^ny(1)&{}_{\widetilde{\eps\rightarrow0^+}}&
\sum_{w\in X^*}\scal{\calA(w)\circ f_{|_{q_0}}}{w}\scal{D_n(1-\eps)\; e^{-x_1\log\eps}\;Z_{\minishuffle}\;e^{-x_0\log\eps}}{w},\\
\theta_0^ny(1)&{}_{\widetilde{\eps\rightarrow0^+}}&
\sum_{w\in X^*}\scal{\calA(w)\circ f_{|_{q_0}}}{w}\scal{E_n(1-\eps)\; e^{-x_1\log\eps}\;Z_{\minishuffle}\;e^{-x_0\log\eps}}{w}.
\end{eqnarray*}

\item If the ordinary Taylor expansions of $\partial_z^ny$ and $\theta_0^ny$ exist then
the coefficients of these expansions belong to the algebra of harmonic sums and there exist algorithmically
computable coefficients $a_i,a'_i\in\Z,b_i,b'_i\in\N,c_i,c'_i\in\calZ[\gamma]$ such that
\begin{eqnarray*}
\partial_z^ny(z)=\sum_{k\ge0}d_kz^n&\mbox{and}&d_k\;{}_{\widetilde{k\rightarrow\infty}}\;\sum_{i\ge0}c_ik^{a_i}\log^{b_i}k,\\
\theta_0^ny(z)=\sum_{k\ge0}t_kz^k&\mbox{and}&t_k\;{}_{\widetilde{k\rightarrow\infty}}\;\sum_{i\ge0}c'_ik^{a'_i}\log^{b'_i}k.
\end{eqnarray*}
\end{enumerate}
\end{corollary}

\subsection{Differential realization}\label{realization}

\subsubsection{Differential realization}

\begin{definition}
The {\it Lie rank} of a formal power series $S\in\KXX$ is the dimension of the vector space generated by
\begin{eqnarray*}
\{S\resd\Pi\bv\Pi\in\LKX\},&\mbox{or respectively by}&\{\Pi\resg S\bv\Pi\in\LKX\}.
\end{eqnarray*}
\end{definition}

\begin{definition}\label{annulateur}
Let $S\in\KXX$ and let us put $\Ann(S):=\{\Pi\in\LKX\bv S\resd\Pi=0\}$,
and $\Ann^{\bot}(S):=\{Q\in(\KXX,\shuffle)\bv Q\resd\Ann(S)=0\}$.
\end{definition}

It is immediate that $\Ann^{\bot}(S)\ni S$. It follows then( see
\cite{fliessrealisation,reutenauerrealisation} and Lemma \ref{reslettre}),

\begin{lemma}\label{finitedimension_stable}
Let $S\in\KXX$. Then
\begin{enumerate}
\item If $S$ is of finite Lie rank, $d$, then the dimension of $\Ann^{\bot}(S)$ is $d$.
\item For any $Q_1$ and $Q_2\in\Ann^{\bot}(S)$, one has $Q_1\shuffle Q_2\in\Ann^{\bot}(S)$.
\item For any $P\in\KX$ and $Q_1\in\Ann^{\bot}(S)$, one has $P\resg Q_1\in\Ann^{\bot}(S)$.
\end{enumerate}
\end{lemma}

\begin{definition}{\rm (\cite{fliessrealisation})}
The formal power series $S\in\KXX$ is {\it differentially produced} if there exist
an integer $d$, a power series $f\in\K[\![\bar{q}_1,\ldots,\bar{q}_d]\!]$, a homomorphism $\calA$ from $X^*$ to the algebra of differential operators generated by
\begin{eqnarray*}
\calA(x_i)=\sum_{j=1}^dA_i^j(\bar{q}_1,\ldots,\bar{q}_d)\frac{\partial}{\partial\bar{q}_j},&\mbox{where}&
\forall j=1,\ldots,d,A_i^j(\bar{q}_1,\ldots,\bar{q}_d)\in\K[\![\bar{q}_1,\ldots,\bar{q}_d]\!]
\end{eqnarray*}
such that, for any $w\in X^*$, one has $\pol{S\bv w}=\calA(w)\circ f_{|_0}$.

The pair $(\calA,f)$ is called the {\it differential representation} of $S$ of dimension $d$.
\end{definition}

\begin{proposition}{\rm (\cite{reutenauerrealisation})}\label{direct}
Let $S\in\KXX$. If $S$ is differentially produced then 
it satisfies the growth condition and its Lie rank is finite.
\end{proposition}

\begin{proof}
Let $(\calA,f)$ be a differential representation of $S$ of dimension $d$.
Then, by the notations of Definition \ref{fliess}, we get
\begin{eqnarray*}
\sigma f_{|_0}=S=\sum_{w\in X^*}(\calA(w)\circ f)_{|_0}\;w.
\end{eqnarray*}
We put
\begin{eqnarray*}
\forall j=1,\ldots,d,&&T_j=\sum_{w\in X^*}\frac{\partial(\calA(w)\circ f)}{\partial\bar{q}_j}\;w.
\end{eqnarray*}

Firstly, by Theorem \ref{growthcondition}, the generating series $\sigma f$ satisfies the growth condition.

Secondly, for any $\Pi\in\LKX$ and for any $w\in X^*$, one has
\begin{eqnarray*}
\langle\sigma f\resd\Pi\bv w\rangle=\langle\sigma f\bv \Pi w\rangle=\calA(\Pi w)\circ f
=\calA(\Pi)\circ(\calA(w)\circ f).
\end{eqnarray*}
Since $\calA(\Pi)$ is a derivation over $\K[\![\bar{q}_1,\ldots,\bar{q}_d]\!]$:
\begin{eqnarray*}
\calA(\Pi)
&=&\sum_{j=1}^d(\calA(\Pi)\circ\bar{q}_j)\frac{\partial}{\partial\bar{q}_j},\\
\Rightarrow\quad
\calA(\Pi)\circ(\calA(w)\circ f)
&=&\sum_{j=1}^d(\calA(\Pi)\circ\bar{q}_j)\frac{\partial(\calA(w)\circ f)}{\partial\bar{q}_j}
\end{eqnarray*}
then we deduce that
\begin{eqnarray*}
\forall w\in X^*,\quad\langle\sigma f\resd\Pi\bv w\rangle
&=&\sum_{j=1}^d(\calA(\Pi)\circ\bar{q}_j)\langle T_j\bv w\rangle,\\
\iff\quad\sigma f\resd\Pi
&=&\sum_{j=1}^d(\calA(\Pi)\circ\bar{q}_j)\;T_j.
\end{eqnarray*}
This means that $\sigma f\resd\Pi$ is a $\K$-linear combination
of $\{T_j\}_{j=1,\ldots,d}$ and the dimension of the vector space
$\mathrm{span}\{\sigma f\resd\Pi\bv\Pi\in\LKX\}$ is less than or equal to $d$.
\end{proof}

\subsubsection{Fliess' local realization theorem}

\begin{proposition}{\rm (\cite{reutenauerrealisation})}\label{gcdp}
Let $S\in\KXX$ with Lie rank $d$. 
Then there exists a basis ${S}_1,\ldots,{S}_d\in\KXX$ of
$(\Ann^{\bot}(S),\shuffle)\cong(\K[\![{S}_1,\ldots,{S}_d]\!],\shuffle)$
such that the ${S}_i$'s are proper and for any $R\in\Ann^{\bot}(S)$, one has
\begin{eqnarray*}
R=\sum_{i_1,\ldots,i_d\ge0}\frac{r_{i_1,\ldots,i_n}}{i_1!\ldots i_d!}
{S}_1^{\minishuffle i_1}\shuffle\ldots\shuffle{S}_d^{\minishuffle i_d},
&\mbox{where}&r_{0,\ldots,0}=\scal{R}{1_{X^*}},r_{i_1,\ldots,i_d}\in\K.
\end{eqnarray*}
\end{proposition}

\begin{proof}
By Lemma \ref{finitedimension_stable}, such a basis exists.
More precisely, since the Lie rank of $S$ is $d$ then there exist
$P_1,\ldots,P_d\in\LKX$ such that
$S\resd P_1,\ldots,S\resd P_d\in(\KXX,\shuffle)$ are $\K$-linearly independent.
By duality, there exists ${S}_1,\ldots,{S}_d\in(\KXX,\shuffle)$ such that
\begin{eqnarray*}
\forall i,j=1,\ldots,d,\quad\langle{S}_i\bv P_j\rangle=\delta_{i,j},
&\mbox{and}&R=\prod_{i=1}^d\exp({S}_i\;P_i).
\end{eqnarray*}
Expanding this product, one obtains, via PBW theorem, the expected expression
for the coefficients $\{r_{i_1,\ldots,i_d}=\scal{R}{P_1^{i_1}\ldots P_d^{i_d}}\}_{i_1,\ldots,i_d\ge0}$.
Hence, $(\Ann^{\bot}(S),\shuffle)$ is generated by ${S}_1,\ldots,{S}_d$.
\end{proof}

With the notations of Proposition \ref{gcdp}, one has

\begin{corollary}\label{polynomialrealization}
\begin{enumerate}
\item If $S\in\K[{S}_1,\ldots,{S}_d]$ then, for any $i=0,1$ and for any $j=1,\ldots,d$,
one has $x_i\resg S\in\Ann^{\bot}(S)=\K[{S}_1,\ldots,{S}_d]$.

\item The power series $S$ satisfies the growth condition if and only if,
for any $i=1,\ldots,d$, ${S}_i$ also satisfies the growth condition.
\end{enumerate}
\end{corollary}

\begin{proof}
Assume there exists $j\in[1,\ldots,d]$ such that ${S}_j$ does not satisfy the growth condition.
Since $S\in\Ann^{\bot}(S)$ then using the decomposition of $S$ on ${S}_1,\ldots,{S}_d$,
one obtains a contradiction with the fact that $S$ satisfies the growth condition.

Conservely, using Proposition \ref{shufflegrowth}, we get the expected results.
\end{proof}

\begin{theorem}{\rm (\cite{fliessrealisation})}
The formal power series $S\in\KXX$ is differentially produced
if and only if its Lie rank is finite and if it satifies the $\chi$-growth condition.
\end{theorem}

\begin{proof}
By Proposition \ref{direct}, one gets a direct proof. Conversely, since the Lie rank of $S$ equals $d$
then by Proposition \ref{gcdp}, setting $\sigma f_{|_0}=S$ and, for $j=1,\ldots,d$, $\sigma\bar{q}_i={S}_i$,
\begin{enumerate}
\item We choose the observation $f$ as follows
\begin{eqnarray*}
f(\bar{q}_1,\ldots,\bar{q}_d)=\sum_{i_1,\ldots,i_d\ge0}
\frac{r_{i_1,\ldots,i_n}}{i_1!\ldots i_d!}
\bar{q}_1^{i_1}\ldots\bar{q}_d^{i_d}
\in\K[\![\bar{q}_1,\ldots,\bar{q}_d]\!]
\end{eqnarray*}
such that
\begin{eqnarray*}
\sigma f_{|_0}(\bar{q}_1,\ldots,\bar{q}_d)
=\sum_{i_1,\ldots,i_d\ge0}\frac{r_{i_1,\ldots,i_n}}{i_1!\ldots i_d!}
(\sigma \bar{q}_1)^{\minishuffle i_1}\shuffle\ldots\shuffle(\sigma \bar{q}_d)^{\minishuffle i_d},
\end{eqnarray*}

\item It follows that, for $i=0,1$ and for $j=1,\ldots,d$, the residual $x_i\resg\sigma\bar{q}_j$
belongs to $\Ann^{\bot}(\sigma f_{|_0})$ (see also Lemma \ref{finitedimension_stable}),

\item Since $\sigma f$ satisfies the $\chi$-growth condition then,
the generating series $\sigma\bar{q}_j$ and $x_i\resg\sigma\bar{q}_j$ (for $i=0,1$ and for $j=1,\ldots,d$)
verify also the growth condition. We then take (see Lemma \ref{fields})
\begin{eqnarray*}
\forall i=0,1,\forall j=1,\ldots,d,&&
\sigma A_j^i(\bar{q}_1,\ldots,\bar{q}_d)=x_i\resg\sigma\bar{q}_j,
\end{eqnarray*}
by expressing $\sigma A_j^i$ on the basis $\{\sigma\bar{q}_i\}_{i=1,\ldots,d}$ of $\Ann^{\bot}(\sigma f_{|_0})$,

\item The homomorphism $\calA$ is then determined as follows
\begin{eqnarray*}
\forall i=0,1,&&\calA(x_i)
=\sum_{j=0}^dA_j^i(\bar{q}_1,\ldots,\bar{q}_d)\frac{\partial}{\partial \bar{q}_j},
\end{eqnarray*}
where, by Lemma \ref{coefficients}, one has $A_j^i(\bar{q}_1,\ldots,\bar{q}_d)=\calA(x_i)\circ\bar{q}_j$.
\end{enumerate}
Thus, $(\calA,f)$ provides a differential representation\footnote{
In \cite{fliessrealisation,reutenauerrealisation}, the reader can find the discussion
on the {\it minimal} differential representation.} of dimension $d$ of $S$.
\end{proof}

Moreover, one also has the following

\begin{theorem}{\rm (\cite{fliessrealisation})}
Let $S\in\KXX$ be a differentially produced formal power series.
Let $(\calA,f)$ and $(\calA',f')$ be two differential representations
of dimension $n$ of $S$. There exist a continuous and convergent automorphism $h$ of $\K$ such that
\begin{eqnarray*}
\forall w\in X^*,\forall g\in\K,\quad h(\calA(w)\circ g)=\calA'(w)\circ(h(g))
&\mbox{and}&f'=h(f).
\end{eqnarray*}
\end{theorem}

Since any rational power series satisfies the growth condi\-tion and its Lie rank is less than or equal
to its Hankel rank which is finite \cite{fliessrealisation} then

\begin{corollary}\label{polynomialrealisation}
Any rational power series and any polynomial over $X$
with coefficients in $\K$ are differentially produced.
\end{corollary}

\end{document}

%% file: macrosB.tex
\usepackage{amsmath}
\usepackage{amsfonts}
\usepackage{amssymb}

\newcommand{\calA}{{\mathcal A}}
\newcommand{\calZ}{{\mathcal Z}}
\newcommand{\calL}{{\mathcal L}}
\newcommand{\calM}{{\mathcal M}}
\newcommand{\calH}{{\mathcal H}}
\newcommand{\calC}{{\mathcal C}}

\newcommand{\N}{{\mathbb N}}
\renewcommand{\P}{\N_{>0}}
\newcommand{\Z}{{\mathbb Z}}
\newcommand{\Q}{{\mathbb Q}}
\newcommand{\R}{{\mathbb R}}
\newcommand{\C}{{\mathbb C}}
\newcommand{\K}{{\mathbb K}}

\newcommand{\Frac}[2]{\displaystyle \frac{#1}{#2}}
\newcommand{\Sum}[2]{\displaystyle{\sum_{#1}^{#2}}}
\newcommand{\Prod}[2]{\displaystyle{\prod_{#1}^{#2}}}

\newcommand{\Lim}[1]{\displaystyle{\lim_{#1}\ }}

\newcommand{\eps}{{\varepsilon}}

\def\pol#1{\langle #1 \rangle}
\def\Lyn{{\mathcal Lyn}}

\def\llm#1{\mathcal #1}

\def\CX{\C \langle X \rangle}

\def\abs#1{|#1|}
\def\absv#1{\|#1\|}
\def\Conv{{\rm CV}}

 \def\shuffle{\mathop{_{^{\sqcup\!\sqcup}}}} 

\gdef\stuffle{\;%
  \setlength{\unitlength}{0.0125cm}%
  \begin{picture}(20,10)(220,580) 
  \thinlines 
  \put(220,592){\line( 0,-1){ 10}} 
  \put(220,582){\line( 1, 0){ 20}} 
  \put(240,582){\line( 0, 1){ 10}} 
  \put(230,592){\line( 0,-1){ 10}} 
  \put(225,587){\line( 1, 0){ 10}} 
  \end{picture}\; 
}

\newcommand{\Li}{\operatorname{Li}}
\newcommand{\ad}{\operatorname{ad}}

\def\L{\mathrm{L}}
\def\H{\mathrm{H}}

\def\P{\mathrm{P}}

\def\Ann{\mathrm{Ann}}
\def\F{\mathrm{F}}

\def\deg{\mathrm{deg}}
\def\wgt{\mathrm{wgt}}

\newcommand{\poly}[2]{#1 \langle #2 \rangle}

\def\QX{\poly{\Q}{X}}

\def\CX{\poly{\C}{X}}

\def\QY{\poly{\Q}{Y}}

\newcommand{\serie}[2]{#1 \langle \! \langle #2 \rangle \! \rangle}

\def\KXX{\serie{\K}{X}}

\def\QXX{\serie{\Q}{X}}
\def\CcXX{\serie{\C^{\mathrm{cont}}}{X}}

\def\CXX{\serie{\C}{X}}

\def\QY_0{\Q\left\langle{Y_0}\right\rangle}
\def\pol#1{\langle #1 \rangle}
\def\ser#1{\langle\!\langle #1 \rangle\!\rangle}
\def\AX{A \langle X \rangle}
\def\KX{\K \langle X \rangle}
\def\AY{A \langle Y \rangle}

\newcommand{\calT}{\mathcal{T}}

\def\Lie{{\cal L}ie}

\def\LAX{\Lie_{A} \langle X \rangle}

\def\LQX{\Lie_{\Q} \langle X \rangle}

\def\LKX{\Lie_{\K} \langle X \rangle}

\def\LXX{\Lie_{\C} \langle\!\langle X \rangle \!\rangle}

\def\CX{\C \langle X \rangle}
\def\CXX{\serie{\C}{X}}

\def\Lim{\displaystyle\lim}
\def\Sum{\displaystyle\sum}
\def\Prod{\displaystyle\prod}
\def\Frac{\displaystyle\frac}
\def\path{\rightsquigarrow}
\def\reg{\mathop\mathrm{reg}\nolimits}

\def\resg{\triangleleft}
\def\resd{\triangleright}
\def\bv{\mid}
\def\bbv{\bv\!\bv}
\def\abs#1{\bv\!#1\!\bv}

\gdef\minishuffle{{\scriptstyle \shuffle}}  
\gdef\ministuffle{{\scriptstyle \stuffle}}

\def\rd{\triangleright}
\def\rg{\triangleleft}
\def\deg{\mathop\mathrm{deg}\nolimits}
\def\w{\mathop\mathrm{\omega}\nolimits}
\def\supp{\mathop\mathrm{supp}\nolimits}

\def\binom#1#2{{#1\choose#2}}


%% file: newpack3.tex
\usepackage{color}
\usepackage{ulem}

\def\scal#1#2{\langle #1\bv#2 \rangle}
\def\ncp#1#2{#1\langle #2\rangle}
